\documentclass[a4paper,11pt]{article}
\usepackage{amsmath,amssymb,graphicx,amsthm}

\usepackage[hidelinks]{hyperref}
\usepackage{algorithm}
\usepackage{algpseudocode}
\usepackage{subcaption}
\usepackage{tikz}
\usetikzlibrary{calc}
\usepackage{tkz-berge}
\usetikzlibrary{arrows.meta}
\usepackage{placeins}
\usepackage{array}
\usetikzlibrary{arrows.meta,decorations.markings,decorations}
\usepackage{mathtools}
\usepackage{multirow}
\usepackage{enumitem}
\usepackage{makecell}
\usepackage{multicol}

\usepackage{pgfplots}
\pgfplotsset{width=8.5cm}
\pgfplotsset{compat=1.8}

\usepackage[
top    = 2.675cm,
bottom = 3.4cm,
left   = 2.5cm,
right  = 2.5cm]{geometry}

\usepackage[margin=1cm]{caption}

\newtheorem{THM}{THEOREM}[section]

\newtheorem{theorem}[THM]{Theorem}

\newtheorem{claim}[THM]{Claim}
\newtheorem{corollary}[THM]{Corollary}

\theoremstyle{definition}
\newtheorem{remark}{Remark}[section]
\newtheorem{problem}{Problem}   

\newcommand{\FBOX}{\hspace*{\fill}$\rule{0.19cm}{0.19cm}$}

\newenvironment{subproof}[1][\proofname]{%
  \begin{proof}[#1]%
}{%
  \end{proof}%
}

\usepackage[textsize=tiny]{todonotes}

\def\R{\mathbb R}

\def\Z{\mathbb Z}

\def\1{\mathbb 1}

\def\niceArrow{-{Stealth[length=2.25mm]}}

\usepackage{graphicx,accents}
\usepackage{esvect}
\makeatletter
\DeclareRobustCommand{\cev}[1]{%
  {\mathpalette\do@cev{#1}}%
}
\newcommand{\do@cev}[2]{%
  \vbox{\offinterlineskip
    \sbox\z@{$\m@th#1 x$}%
    \ialign{##\cr
      \hidewidth\reflectbox{$\m@th#1\vec{}\mkern0mu$}\hidewidth\cr
      \noalign{\kern-\ht\z@}
      $\m@th#1#2$\cr
    }
  }
}
\makeatother

\newcolumntype{x}[1]{>{\centering\let\newline\\\arraybackslash\hspace{0pt}}m{#1}}

\usetikzlibrary{snakes}
\tikzset{wavy/.style={decorate,decoration={snake,amplitude=.4mm,segment length=2mm,post length=0mm,pre length=0mm},line width=.5}}

\tikzset{paralleledge/.style={to path={
      \pgfextra{%
        \pgfmathsetmacro{\startf}{-(#1-1)/2}
        \pgfmathsetmacro{\endf}{-\startf}
        \pgfmathsetmacro{\stepf}{\startf+1}}
      \ifnum 1=#1 -- (\tikztotarget)  \else
        let \p{mid}=($(\tikztostart)!0.5!(\tikztotarget)$)
        in
        \foreach \i in {\startf,\stepf,...,\endf}
        {%
          (\tikztostart) .. controls ($ (\p{mid})!\i*6pt!90:(\tikztotarget) $) .. (\tikztotarget)
        }
      \fi
      \tikztonodes
    }
  }
}

\tikzset{circleAroundEdges/.style n args={3}{
    decorate, decoration = {markings, mark=at position #3 with {\draw[double=black,white,double distance=.6pt] (0,#2) arc [x radius = #1, y radius = #2, start angle = 90, end angle = -125];} , mark=at position #3 with{\draw[double=black,white,double distance=.6pt] (0,-#2) arc [x radius = #1, y radius = #2, start angle = 270, end angle = 360+125];}}
  }
}

\usepackage{titling}
\usepackage[hang,flushmargin]{footmisc}
\usepackage[yyyymmdd]{datetime}

\begin{document}

\title{\vspace{-15mm}Vertex-ordering and arc-partitioning problems}

\author{N\'ora A.\ Borsik \thanks{Department of Operations Research, E\"otv\"os Lor\'and University, P\'azm\'any P.\ s.\ 1/c, Budapest, Hungary. E-mail: {\tt nborsik@gmail.com}}
  \and P\'eter Madarasi \thanks{HUN-REN Alfr\'ed R\'enyi Institute of Mathematics, and Department of Operations Research, E\"otv\"os Lor\'and University, P\'azm\'any P.\ s.\ 1/c, Budapest, Hungary. E-mail: {\tt madarasip@staff.elte.hu} (corresponding author)}
}

\date{\vspace*{-44pt}}

\maketitle

\begin{abstract}
  We study vertex-ordering problems in loop-free digraphs subject to constraints on the left-going arcs, focusing on existence conditions and computational complexity.
  As an intriguing special case, we explore vertex-specific lower and upper bounds on the left-outdegrees and right-indegrees.
  We show, for example, that deciding whether the left-going arcs can form an in-branching is solvable in polynomial time and provide a necessary and sufficient condition, while the analogous problem for an in-arborescence turns out to be NP-complete.
  We also consider a weighted variant that enforces vertex-specific lower and upper bounds on the $w$\nobreakdash-weighted left-outdegrees, which is particularly relevant in applications.
  Furthermore, we investigate the connection between ordering problems and their arc-partitioning counterparts, where one seeks to partition the arcs into a subgraph from a specific digraph family and an acyclic subgraph ---  equivalently, one seeks to cover all directed cycles with a subgraph belonging to a specific family.
  For the family of in-branchings, unions of disjoint dipaths, and matchings, the two formulations coincide, whereas for in-arborescences, dipaths, Hamiltonian dipaths, and perfect matchings the formulations diverge.
  Our results yield a comprehensive complexity landscape, unify diverse special cases and variants, clarify the algorithmic boundaries of ordered digraphs, and relate them to broader topics including graph degeneracy, acyclic orientations, influence propagation, and rank aggregation.\\

\end{abstract}

\section{Introduction}\label{sec:introduction}

We study ordering and partitioning problems in digraphs, motivated by classical notions such as graph degeneracy and acyclic orientations.
Our central question is whether the vertices of a digraph can be ordered so that the set of left-going arcs forms a subgraph with prescribed structural properties.
We formalize this through four fundamental problems.

\begin{problem}[Vertex ordering for a digraph family $\mathcal{F}$]\label{prob:vertexOrdering}
  Given a loop-free digraph $D=(V,A)$ and a family $\mathcal{F}$ of digraphs, decide whether the vertices of $D$ can be ordered such that the set of left-going arcs forms a subgraph belonging to $\mathcal{F}$.
\end{problem}

This problem captures a wide range of natural problems, which show a rich complexity landscape.
For example, we prove that the case when $\mathcal{F}$ is the family of in-branchings can be solved in polynomial time, while the case of in-arborescences is NP-complete.

The following particularly interesting special case arises when $\mathcal{F}$ consists of indegree- and outdegree-bounded (acyclic) digraphs.
\begin{problem}[Indegree- and outdegree-bounded ordering]\label{prob:degreeBounded}
  Given a loop-free digraph $D=(V,A)$, lower bound functions $f_{\delta}, f_{\varrho} : V \to \Z_+ \cup \{-\infty\}$, and upper bound functions $g_{\delta}, g_{\varrho} : V \to \Z_+ \cup \{+\infty\}$, decide whether there exists an ordering of the vertices such that
  \[
    f_{\delta}(v) \le \delta^{\ell}(v) \le g_{\delta}(v) \hspace{1cm} \text{ and } \hspace{1cm} f_{\varrho}(v) \le \varrho^{r}(v) \le g_{\varrho}(v)
  \]
  hold for each $v\in V$, where $\delta^{\ell}(v)$ and $ \varrho^{r}(v)$ denote the left-outdegree and right-indegree of $v$, respectively.
\end{problem}

Observe that any vertex-ordering problem with simultaneous bounds for the left- and right-outdegrees and the left- and right-indegrees can be easily reduced to Problem~\ref{prob:degreeBounded}.

To broaden the scope of applicability, we introduce a weighted version of left-outdegree bounds, which generalizes the in-branching case and naturally connects to rank aggregation.
\begin{problem}[$(f,g;\sum w)$-bounded ordering]\label{prob:fgBounded}
  Given a loop-free digraph $D = (V, A)$, a lower bound function $f: V \to \R_+ \cup \{-\infty\}$, an upper bound function $g: V \to \R_+ \cup \{+\infty\}$, and a weight function $w: A \to \R_+ \cup \{+\infty\}$, decide whether there exists an ordering of the vertices such that
  \[
    f(v) \leq \delta^{\ell}_w(v) \leq g(v),
  \]
  holds for each $v\in V$, where $\delta^{\ell}_w(v)$ denotes the weighted left-outdegree of $v$.
  For $w \equiv 1$, the problem is referred to as the \emph{$(f, g)$-bounded ordering problem}.
\end{problem}

Finally, any feasible vertex order for Problem~\ref{prob:vertexOrdering} partitions the arcs into left- and right-going arcs, which leads us to the following arc-partitioning problem.
\begin{problem}[Arc partitioning for a digraph family $\mathcal{F}$]\label{prob:arcPartition}
  Given a loop-free digraph $D=(V,A)$ and a family $\mathcal{F}$ of digraphs, decide whether $A$ can be partitioned into two parts: one belonging to $\mathcal{F}$ and an acyclic subgraph.
\end{problem}
For several natural families of acyclic digraphs --- such as in-branchings, matchings (directed arbitrarily), and unions of disjoint dipaths --- the vertex-ordering and arc-partitioning formulations coincide.
For others --- including in-arborescences, perfect matchings (directed arbitrarily), dipaths, and Hamiltonian dipaths ---, the two perspectives diverge, leading to diverse computational behaviors.

Note that Problem~\ref{prob:arcPartition} is equivalent to deciding whether the directed cycles in $D$ can be covered with a subgraph belonging to the family $\mathcal{F}$.

\medskip

Now we provide a brief overview of problems related to the subject of this paper.

\paragraph{Degeneracy of graphs}
An undirected graph $G = (V, E)$ is called \emph{$k$\nobreakdash-degenerate} if, for every subset $V' \subseteq V$, the minimum degree in the induced subgraph $G[V']$ is at most $k$~\cite{lick1970k}.
The \emph{degeneracy} of a graph $G$ is the smallest number $k$ for which $G$ is $k$\nobreakdash-degenerate.
It is well-known that the degeneracy of a graph can be computed in linear time~\cite{matula1983smallest}.
Notice that a graph $G$ is $k$\nobreakdash-degenerate if and only if there exists an ordering of the vertices such that the left-degree of each vertex is at most $k$.
An undirected graph can be considered as a symmetric digraph, for which the $(-\infty, g)$\nobreakdash-bounded ordering problem with $g \equiv k$ is solvable if and only if the graph is $k$-degenerate.
The $(-\infty, g)$\nobreakdash-bounded ordering problem can thus be seen as a generalization of degeneracy for digraphs.

\paragraph{Degree-constrained acyclic orientation problem}
In~\cite{kiraly2018acyclic}, the authors studied a problem closely related to the $(f,g)$\nobreakdash-bounded ordering problem.
Given an undirected graph $G = (V, E)$ and two functions $f': V \to \Z_+$ and $g': V \to \Z_+$, where $f'(v) + g'(v) \leq d(v)$ for each vertex $v \in V$ (with $d(v)$ denoting the degree of $v$), the goal is to determine whether $G$ has an acyclic orientation such that $f'(v) \leq \varrho(v) \leq d(v) - g'(v)$ for each $v \in V$, where $\varrho(v)$ represents the indegree of vertex $v$ in the acyclic orientation.
Such an orientation is referred to as an \emph{$(f', g')$\nobreakdash-bounded acyclic orientation} of $G$.
Note that this problem is equivalent to finding an ordering of the vertices such that the left-degree of each vertex $v$ is bounded below by $f'(v)$ and above by $(d(v) - g'(v))$.

Considering undirected graphs as symmetric digraphs, the $(f,g)$\nobreakdash-bounded ordering problem generalizes the $(f', g')$\nobreakdash-bounded acyclic orientation problem to digraphs.
Specifically, $G$ has an $(f', g')$\nobreakdash-bounded acyclic orientation if and only if the corresponding symmetric digraph $D$ has an $(f,g)$\nobreakdash-bounded ordering, where $f \equiv f'$ and $g \equiv d - g'$.
The topological order of an $(f', g')$\nobreakdash-bounded acyclic orientation of $G$ corresponds to an $(f,g)$\nobreakdash-bounded ordering of $D$, and the arcs going from right to left in the $(f,g)$\nobreakdash-bounded order of $D$ form an $(f', g')$\nobreakdash-bounded acyclic orientation of $G$.
The $(f', g')$\nobreakdash-bounded acyclic orientation problem was proven to be NP-complete~\cite{kiraly2018acyclic}.
This immediately implies the following.
\begin{corollary}\label{cor:LowerAndUpper}
  The $(f,g)$\nobreakdash-bounded ordering problem is NP-complete even when restricted to symmetric digraphs.
  \FBOX
\end{corollary}

However, the $(f', g')$\nobreakdash-bounded acyclic orientation problem was shown to be solvable in polynomial time in certain special cases, such as when $f'(v) g'(v) = 0$ for all $v \in V$ (i.e., each vertex has either a lower bound or an upper bound on its indegree), or $f'(v) = d(v) - g'(v)$ for each $v \in V$ (i.e., each vertex has an exact specification for its indegree)~\cite{kiraly2018acyclic}.
They also examined the complexity of the $(f', g')$\nobreakdash-bounded acyclic orientation problem in the case when $f'(v) = k$ and $g'(v) = \ell$ for positive integers $k$ and $\ell$.
For $k = \ell = 1$, the goal is to find an ordering such that each vertex has at least one edge going to the left and at least one edge going to the right.
This problem is known as the \emph{$s$\nobreakdash-$t$ numbering problem}, which is solvable in polynomial time~\cite{lempel1967algorithm}.
The analogous problem for digraphs is also solvable in polynomial time~\cite{cheriyan1994directeds}, but the more general \emph{betweenness problem} is NP-hard~\cite{opatrny1979total}.
For $k = \ell = 2$, the $(f', g')$\nobreakdash-bounded acyclic orientation problem becomes NP-complete~\cite{kiraly2018acyclic}, but the complexity of the case where $k = 1$ and $\ell = 2$ remains open.

All of the results mentioned above directly apply to the $(f,g)$\nobreakdash-bounded ordering problem in the case of symmetric digraphs.
In this paper, we investigate the complexities of similar special cases of the $(f,g)$\nobreakdash-bounded ordering problem.

\paragraph{Arc-partitioning and vertex-ordering problems}
Arc-partitioning problems have been extensively studied in the literature~\cite{bang2022complexity, bang2015restricted, bang2020arc}.
For instance, partitioning a digraph into a directed cycle and an acyclic subgraph, or into a directed 2-factor and an acyclic subgraph, are known to be NP-complete problems with respect to Turing reduction~\cite{bang2022complexity}.
Similarly, determining whether a digraph contains an $r$\nobreakdash-in-arborescence and an $r$\nobreakdash-out-arborescence for a given root $r$ that are arc-disjoint is NP-complete~\cite{bang1991edge}.
In other words, the problem of partitioning a digraph into a subgraph containing an $r$\nobreakdash-in-arborescence and another containing an $r$\nobreakdash-out-arborescence is NP-complete.
However, the problem becomes solvable when considering two arc-disjoint $r$\nobreakdash-out-arborescences or $k$ arc-disjoint $r$\nobreakdash-out-arborescences in general~\cite{edmonds1973edge}.
In this work, we consider problems involving partitioning the arc set into a member of a specified digraph family (such as in-branchings, in-arborescences, matchings, or dipaths).
The analogous problems can be defined for undirected graphs as well, where the goal is to partition the edge set into a member of a specified undirected graph family and into a forest --- which is an undirected analogue of an acyclic digraph.
Many of the corresponding problems for undirected graphs were considered in~\cite{bernath2015tractability}.
For example, they proved that it is NP-hard to partition into a path and a forest, or into a cycle and a forest.
However, partitioning into two forests, or into a spanning tree and a forest are polynomial-time solvable problems~\cite{kameda1973, kishi1969maximally}.
In~\cite{montassier2012decomposing}, the authors provided a sufficient condition for partitioning into a matching and a forest.
In~\cite{yang2018decomposing}, this condition was generalized for partitioning into a matching and $k$ forests for a positive integer $k$.
Some of the directed partitioning problems can also be viewed as partitioning into an in- or outdegree bounded digraph and an acyclic subgraph.
In~\cite{wood2004bounded}, the authors considered a similar degree-bounded acyclic decomposition problem.
They proved that, for any integer $k\ge 2$, every simple digraph can be partitioned into $k$ acyclic subgraphs such that each outdegree is at most $\left\lceil \frac{\delta(v)}{k-1}\right\rceil$.

A well-known related problem is the \emph{feedback arc set problem} asking for the fewest number of arcs whose removal makes the digraph acyclic, which is known to be NP-hard~\cite{karp1972reducibility}.
This problem is equivalent to finding a vertex ordering that minimizes the number of left-going arcs.

Another relevant problem is the \emph{minimum target set selection problem}, which models the propagation of influence in a network~\cite{charikar2016approximating, chen2009approximability}.
In this problem, the network is represented by a (directed) graph, and each vertex $v$ has a threshold $\tau(v) \in \Z_+$.
An initial subset of vertices is activated, and in each round, a vertex $v$ is activated if at least $\tau(v)$ of its (in-)neighbors are already active.
Given an initial set of activated vertices, we can determine whether the entire network will be activated using the $(f,+\infty)$\nobreakdash-bounded ordering problem for the reversed digraph --- even in the natural arc-weighted variant by relying on the $(f,+\infty;\sum w)$\nobreakdash-bounded ordering problem.
The bounds are set such that $f(v) = 0$ for vertices in the initial set and $f(v) = \tau(v)$ for all other vertices.
This yields an ordering of the vertices that allows them to become activated one-by-one.

However, the problem of finding a minimum-size initial set that activates the entire network is hard to approximate within a ratio of $O(2^{\log^{1-\varepsilon}n})$ for any $\varepsilon > 0$, even when $\tau \equiv 2$~\cite{chen2009approximability}.
It is also hard to approximate within a ratio of $O(n^{\frac{1}{2}-\varepsilon})$ for any $\varepsilon > 0$, assuming the Planted Dense Subgraph Conjecture~\cite{charikar2016approximating}.

\paragraph{Rank aggregation problems}\label{sec:rankAggregation}
Consider a competition in which different judges provide complete rankings of candidates, and our goal is to determine a common ranking that represents a ``fair'' consensus of the preferences of the judges.
In the \emph{Kemeny rank aggregation problem}, the distance between two rankings is defined as the number of pairs of candidates whose order is reversed between the two rankings~\cite{kemeny1959mathematics}.
The goal is to find a common ranking that minimizes the total distance from the rankings of the judges.
Another variant of the problem aims to find a ranking that is closest to the farthest ranking, minimizing the maximum distance between the common ranking and the rankings of the judges.
Both of these problems are known to be NP-hard~\cite{biedl2009complexity}.

We now introduce a related problem where the distance is measured from the perspective of the candidates rather than the judges.
For a given candidate $v$, let $\varphi(v)$ denote the number of candidates that are ranked higher than $v$ in the common ranking, but lower than $v$ in the majority of the rankings of the judges.
This measure quantifies how ``unfair'' the common ranking appears from the perspective of candidate $v$.
The goal is to find a common ranking that minimizes the maximum $\varphi(v)$ across all candidates.

To reduce this problem to the $(f,g)$\nobreakdash-bounded ordering problem, we introduce a \emph{penalty digraph} in which each vertex corresponds to a candidate.
There is a directed arc from vertex $u$ to vertex $v$ if the majority of the judges rank $u$ before $v$.
If we order the vertices by an arbitrary ranking, then the left-outdegree of vertex $v$ --- which is the number of arcs from $v$ to vertices that precede it in the ordering~--- corresponds directly to the distance $\varphi(v)$ according to the candidate $v$.
Thus, the original problem can be reformulated as finding a vertex ordering that minimizes the maximum left-outdegree across all vertices.
This problem can be solved by determining the smallest positive integer $c$ for which the $(f,g)$\nobreakdash-bounded ordering problem has a feasible solution with $f \equiv -\infty$ and $g \equiv c$.

As a natural generalization of the previous problem, each candidate $v$ assigns a ``disappointment score'' $w_v(u)$ to every other candidate $u$, which measures how much $v$ is disappointed if $u$ precedes $v$ in the common order.
For example, $w_v(u)$ could represent the number of judges who rank $v$ before $u$.
The disappointment of $v$ in the common order is then the sum of $w_v(u)$ for all candidates $u$ that precede $v$.
Using the $(-\infty, g; \sum w)$\nobreakdash-bounded ordering problem, we can decide whether there exists an order in which the disappointment of each candidate $v$ is bounded by $g(v)$.
Moreover, we can minimize the maximum disappointment across all candidates in strongly polynomial time, as we will see in Section~\ref{sec:OnlyUpperBounds}.

\paragraph{Our contribution}
In Section~\ref{sec:OnlyUpperBounds}, we investigate Problem~\ref{prob:fgBounded}, namely, the $(f,g;\sum w)$\nobreakdash-bounded ordering problem in the case when either only lower or only upper bounds are given.
We show that the problem remains solvable under these circumstances, and provide necessary and sufficient conditions for the existence of such an order.
In contrast, the problem turns out to be NP-complete when subject to natural modifications, see Section~\ref{sec:Complexity}.
These modifications include relaxing the restriction that arc weights must be non-negative, or when a single vertex is subject to both lower and upper bounds.
Furthermore, we examine the $(f,g)$\nobreakdash-bounded ordering problem with special bound functions.
In particular, we show that the problem becomes NP-complete, for any $a \geq 1$ and $b \geq 2$, with bounds $f(v) = a$ and $g(v) = \delta(v) - b$ (except for the designated first and last vertices) --- in contrast to the solvability of the directed $s$\nobreakdash-$t$ numbering problem~\cite{cheriyan1994directeds}, which corresponds to the case $a=b=1$.
Additionally, we extend this hardness result to the case where $f \equiv g$, meaning that exact bounds are given for the left-outdegrees.
However, the analogous case of the $(f',g')$\nobreakdash-bounded acyclic orientation problem, where exact bounds are given for the indegrees, is known to be solvable in polynomial time~\cite{kiraly2018acyclic}.
Sections~\ref{sec:dDist} and~\ref{sec:lexicographical} investigate two modified versions of the polynomial-time solvable $(-\infty;g)$\nobreakdash-bounded ordering problem.
The first modification introduces a $d$\nobreakdash-distance constraint, where the upper bound $g(v)$ applies only to arcs going from $v$ to the at most $d$ directly preceding vertices.
The second modification explores lexicographical versions of the problem, in which we seek a $(-\infty, g)$\nobreakdash-bounded ordering with either a lexicographically minimal or maximal left-outdegree vector.
Both of these modified problems turn out to be NP-hard.

In Section~\ref{sec:simultaneousBounds}, we provide a comprehensive complexity analysis for Problem~\ref{prob:degreeBounded}, in particular, we consider every case of simultaneous lower, upper, or exact bounds for the left-outdegree and right-indegree of each vertex; see Table~\ref{tab:simultaneous} for a summary of our results.
\setlength{\tabcolsep}{10pt}
\renewcommand{\arraystretch}{1.5}

\begin{table}[h]
  \small
  \begin{center}
    \begin{tabular}{ |c|c|c|c|c|c|c| }
      \hline
      & $\delta^{\ell} \geq f_{\delta}$ & $\delta^{\ell} \leq g_{\delta}$ & $\delta^{\ell}=m_{\delta}$ & $\varrho^{r} \geq f_{\varrho}$ & $\varrho^{r} \leq g_{\varrho}$ & $\varrho^{r}=m_{\varrho}$\\
      \hline
      $\delta^{\ell} \geq f_{\delta}$ & \makecell{in P\\ Thm~\ref{thm:fChar}}& \makecell{NP-c\\ Cor~\ref{cor:LowerAndUpper}} & \makecell{NP-c\\ Cor~\ref{cor:strict}} & \makecell{NP-c\\ Cor~\ref{cor:outLowerInLower}} & \makecell{in P\\ Thm~\ref{thm:outLowerInUpper}} & \makecell{NP-c\\ Cor~\ref{cor:strict}}\\
      \hline
      $\delta^{\ell} \leq g_{\delta}$ &  & \makecell{in P\\ Thm~\ref{thm:gChar}} & \makecell{NP-c\\ Cor~\ref{cor:strict}} & \makecell{in P\\ Thm~\ref{thm:outUpperInLower}} & \makecell{NP-c\\ Thm~\ref{thm:outUpperInUpper}} & \makecell{NP-c\\ Cor~\ref{cor:strict}}\\
      \hline
      $\delta^{\ell}=m_{\delta}$ &  &  & \makecell{NP-c\\ Cor~\ref{cor:strict}} & \makecell{NP-c\\ Cor~\ref{cor:strict}} & \makecell{NP-c\\ Cor~\ref{cor:strict}} & \makecell{in P\\ Thm~\ref{thm:outStrictInStrict}}\\
      \hline
      $\varrho^{r} \geq f_{\varrho}$ &  &  &  & \makecell{in P\\ Thm~\ref{thm:fChar}} & \makecell{NP-c\\ Cor~\ref{cor:LowerAndUpper}} & \makecell{NP-c\\ Cor~\ref{cor:strict}}\\
      \hline
      $\varrho^{r} \leq g_{\varrho}$ &  &  &  &  & \makecell{in P\\ Thm~\ref{thm:gChar}} & \makecell{NP-c\\ Cor~\ref{cor:strict}}\\
      \hline
      $\varrho^{r}=m_{\varrho}$ &  &  &  &  &  & \makecell{NP-c\\ Cor~\ref{cor:strict}}\\
      \hline
    \end{tabular}\caption{The complexity analysis of all vertex-ordering problems with two simultaneous lower bound, upper bound, or prescription for the left-outdegree $\delta^{\ell}$ and right-indegree $\varrho^{r}$ of each vertex.
    }\label{tab:simultaneous}
  \end{center}
\end{table}

We emphasize that, in contrast to the problems where either only the left-outdegrees or only the right-indegrees are exactly prescribed, the strongly restrictive version of the problem when both the left-outdegrees and the right-indegrees are exactly prescribed turns out to be polynomial-time solvable.
This implies the solvability of the following two notable special cases.
First, when the left-going arcs must form an in-arborescence and the right-going arcs an out\nobreakdash-arborescence; second, when the left-going arcs must form an $s$-$t$ Hamiltonian dipath.

In Section~\ref{sec:ArcPartition}, we study Problems~\ref{prob:vertexOrdering} and~\ref{prob:arcPartition} for various natural acyclic families $\mathcal{F}$.
We show that, for in-branchings, unions of disjoint dipaths, and matchings, the vertex-ordering problem is essentially equivalent to the arc-partitioning problem.
However, this is not the case in general: for in-arborescences, perfect matchings, dipaths or Hamiltonian dipaths, the two problems do not coincide.
Table~\ref{tab:partitioningAndOrdering} summarizes the complexities of these problems.
\setlength{\tabcolsep}{5pt}
\renewcommand{\arraystretch}{2}
\begin{table}[h]
\small
  \begin{center}
    \begin{tabular}{|c|c|c|}
      \hline
      \textbf{Digraph family $\mathcal{F}$} & \textbf{Problem~\ref{prob:vertexOrdering}} & \textbf{Problem~\ref{prob:arcPartition}} \\
      \hline
      in-branchings & in P, Cor~\ref{cor:inBranchingOrder} & in P, Thm~\ref{thm:inBranchingPoly} \\
      \hline
      in-arborescences & NP-c, Cor~\ref{cor:inArbOrder} & open \\
      \hline
      matchings & NP-c, Cor~\ref{cor:matchingAcyclic} & NP-c, Thm~\ref{thm:matchingAndAcyclicNPC} \\
      \hline
      perfect matchings & NP-c, Thm~\ref{thm:perfectMatchinOrderNPC} & NP-c, Thm~\ref{thm:perfectMatchingAndAcyclicNPC} \\
      \hline
      unions of disjoint dipaths &
                        {\renewcommand{\arraystretch}{1}
                        \begin{tabular}{c}
                        NP-c, Cor~\ref{cor:disjPathOrder} \\
                        in P for constant number of dipaths, Thm~\ref{thm:kPathOrder}
                        \end{tabular}}
                                                                                            & NP-c, Thm~\ref{thm:disjPathAcyclic} \\
      \hline
      dipaths & in P, Cor~\ref{cor:pathOrder} & NP-c, Thm~\ref{thm:pathAcyclic} \\
      \hline
      Hamiltonian dipaths & in P, Thm~\ref{thm:hamPathOrder} & NP-c, Thm~\ref{thm:hamPathAcyclic} \\
      \hline
    \end{tabular}
    \caption{Complexity results for Problems~\ref{prob:vertexOrdering} and~\ref{prob:arcPartition} for various digraph families $\mathcal{F}$.}\label{tab:partitioningAndOrdering}
  \end{center}
\end{table}

Furthermore, we prove that partitioning into an in-branching and an acyclic subgraph can be solved in polynomial time --- even when some vertices are required to be roots of the in-branching --- along with a necessary and sufficient condition for the existence of such a partition.
This is equivalent to covering all directed cycles with an in-branching, which resembles a related problem where all directed cuts, rather than directed cycles, must be covered by an in-branching~\mbox{\cite[p.~567]{frank2011connections}}.
Although the arc-partitioning problem for in-arborescences remains open, we prove that partitioning into a minimum-cost in-arborescence and an acyclic subgraph is NP-complete.
Furthermore, partitioning into an in-arborescence and a \textit{spanning} acyclic subgraph is NP-complete as a corollary of the hardness proof for Problem~4.2.6 in~\cite{bang2022complexity}.
We also prove that partitioning into a minimum-size in-branching and an acyclic digraph is APX-hard.

For two disjoint subsets $S, T \subseteq V$ of equal size, we can decide in polynomial time whether there exists an ordering of the vertices such that the left-going arcs form $|S|$ disjoint $S$\nobreakdash-$T$ dipaths, and we also provide a necessary and sufficient condition for the existence.
The same problem turns out to be hard if the endpoints of the dipaths are free.

\paragraph{Notation}
Throughout this paper, $G = (V, E)$ denotes a loop-free undirected graph, where $V$ is the set of vertices and $E$ is the set of edges.
The degree of a vertex $v \in V$ in $G$ is denoted by $d(v)$ and the minimum degree in $G$ by $d_{\min}$.
Similarly, let $D = (V, A)$ be a loop-free directed graph (digraph), where $A$ is the set of arcs.
Parallel edges and arcs are allowed.
A weight function $w: A \to \R \cup \{\pm \infty\}$ may be assigned to the arcs in $D$.
The \emph{outdegree} of a vertex $v \in V$ in $D$ is denoted by $\delta(v)$, and its \emph{weighted outdegree} is denoted by $\delta_w(v)$.
The \emph{indegree} of $v$ is $\varrho(v)$, and the \emph{weighted indegree} is $\varrho_w(v)$.
For a subset $V' \subseteq V$, let $D[V']$ denote the subgraph of $D$ induced by $V'$.
The \emph{outdegree} and \emph{weighted outdegree} of $v \in V$ with respect to $V'$ are $\delta(v, V')$ and $\delta_w(v, V')$, respectively.
Similarly, the \emph{indegree} and \emph{weighted indegree} of $v$ with respect to $V'$ are $\varrho(v, V')$ and $\varrho_w(v, V')$.
An \emph{ordering} of the vertices is represented by $\sigma = (\sigma_1, \ldots, \sigma_n)$, where $\sigma_i \in V$ is the vertex that occupies the $i^{\text{th}}$ position in the order.
For a vertex $v = \sigma_i$ in the vertex order $\sigma$, the \emph{left-outdegree} and \emph{weighted left-outdegree} of $v$ in $\sigma$ are given by $\delta(v, \{\sigma_1, \ldots, \sigma_{i-1}\})$ and $\delta_w(v, \{\sigma_1, \ldots, \sigma_{i-1}\})$, respectively, and are denoted by $\delta^{\ell}(v)$ and $\delta^{\ell}_w(v)$.
The \emph{right-outdegree} of $v$ in $\sigma$ is $\delta(v, \{\sigma_{i+1}, \ldots, \sigma_n\})$, and we denote it by $\delta^{r}(v)$.
The \emph{left-indegree} and \emph{right-indegree} of $v$ in $\sigma$ are $\varrho(v, \{\sigma_1, \ldots, \sigma_{i-1}\})$ and $\varrho(v, \{\sigma_{i+1}, \ldots, \sigma_n\})$, respectively, and are denoted by $\varrho^{\ell}(v)$ and $\varrho^{r}(v)$.
Finally, the functions $f: V \to \R \cup \{-\infty\}$, $g: V \to \R \cup \{+\infty\}$, and $m: V \to \R$ represent lower, upper, and exact bounds, respectively, on the degrees or other graph parameters associated with the vertices in~$V$.

\section{Degree-bounded ordering problems}\label{sec:OrderinProblems}

First, we investigate Problem~\ref{prob:fgBounded} along with several natural special cases and modifications.
After that, we move on to Problem~\ref{prob:degreeBounded}, as a straightforward generalization of the $(f,g)$-bounded ordering problem.

\subsection{The $(f,g;\sum w)$-bounded ordering problem}\label{sec:fgBounded}
In Section~\ref{sec:Complexity}, the $(f,g;\sum w)$-bounded ordering problem will be shown to be NP-complete even for simple digraphs.
However, the next section proves that it can be solved in polynomial time provided that $f\equiv -\infty$ or~$g\equiv +\infty$.

\subsubsection{Either lower or upper bounds}\label{sec:OnlyUpperBounds}

By the $(-\infty, g; \sum w)$\nobreakdash-ordering problem, we mean the case where only upper bounds are given, that is, $f \equiv -\infty$.
This section gives a polynomial-time algorithm for solving this problem.
Later, this algorithm and the following theorems will be used to partition a digraph into an in-branching and an acyclic subgraph --- or prove that no such partition exists.
\begin{algorithm}
  \caption{\hspace{0.5cm}\textsc{$(-\infty, g; \sum w)$\nobreakdash-bounded ordering}}\label{alg:UpperBounds}
  \begin{algorithmic}[1]
    \State $V' \coloneqq V$, $n \coloneqq |V|$
    \State Let $\sigma_1, \ldots, \sigma_n$ denote the vertex order we are searching for.
    \For{$i = n, \ldots, 1$}
    \State $V^* \coloneqq \{ v \in V' : \delta_w(v, V'\setminus \{v\}) \leq g(v) \}$\label{alg:line:filter}
    \If{$V^* \neq \emptyset$}
    \State Choose $\sigma_i \in V^*$ arbitrarily.\label{alg:line:select}
    \State $V' \coloneqq V' \setminus \{ \sigma_i \}$
    \Else
    \State \textbf{return} \textit{No solution exists}
    \EndIf
    \EndFor
    \State \textbf{return} $\sigma_1, \ldots, \sigma_n$
  \end{algorithmic}
\end{algorithm}
\FloatBarrier

Algorithm~\ref{alg:UpperBounds} fixes the vertices from right to left.
The set of the non-fixed vertices is denoted by $V'$.
In Line~\ref{alg:line:filter}, the algorithm filters those vertices from $V'$ for which $\delta_w(v, V' \setminus \{v\}) \leq g(v)$.
If at least one such vertex exists, then one of them is selected, placed at the last free position, and deleted from $V'$.
If no such vertex is found, then the algorithm concludes that no solution exists.
Next, we show the correctness of Algorithm~\ref{alg:UpperBounds}.
\begin{theorem}\label{thm:upperBounds}
  Algorithm~\ref{alg:UpperBounds} solves the $(-\infty, g; \sum w)$\nobreakdash-bounded ordering problem.
\end{theorem}
\begin{proof}
  Clearly, the fixed vertices do not violate the upper bounds as $\delta_w(v, V' \setminus \{v\}) \leq g(v)$ holds whenever a vertex $v$ is fixed.
  Thus, if a vertex satisfying this condition can be found in each iteration of the for loop, then the algorithm finds a feasible order for the $(-\infty, g; \sum w)$\nobreakdash-bounded ordering problem.
  Otherwise, no such vertex exists and $V'$ is non-empty.
  Let $\sigma$ be an arbitrary order of $V$, and let $u$ be the last vertex in $V'$ according to $\sigma$.
  Then $\delta^{\ell}_w(u) \geq \delta_w(u,V' \setminus \{u\})> g(u)$ holds, since $u$ is the last vertex from $V'$ and $\delta_w(v, V' \setminus \{v\})>g(v)$ for all $v\in V'$.
  Therefore, $\sigma$ is not feasible, and thus no feasible vertex order exists.
  Consequently, the algorithm correctly finds a feasible vertex order or concludes that no such order exists, and its time complexity is polynomial.
\end{proof}

Observe that this algorithm generalizes to the case where we seek a $(-\infty, g; \sum w)$\nobreakdash-bounded ordering that respects a given partial order $(V, \prec)$ on the vertices.
To enforce the partial order, simply add a new arc $uv$ with weight $w(uv) = +\infty$ for every pair of distinct vertices $u, v \in V$ such that $u \prec v$.

Furthermore, using binary search, one can minimize the maximum weighted left-outdegree $\delta^{\ell}_w(v)$ across all vertices by repeatedly solving the problem for uniform upper bounds $g$.
In fact, the algorithm can be made strongly polynomial.
\begin{remark}\label{remark:minMaxDegree}
  An order $\sigma$ that minimizes the maximum weighted left-outdegree $\delta^{\ell}_w(v)$ over all $v \in V$ can be found in strongly polynomial time by modifying Line~\ref{alg:line:select} of Algorithm~\ref{alg:UpperBounds} as follows: choose $\sigma_i = \arg \min \{\delta_w(v, V' \setminus \{v\}): v \in V^*\}$ instead of arbitrarily selecting $\sigma_i \in V^*$.
  The proof of the correctness is similar to that of Theorem~\ref{thm:upperBounds}.
  $\bullet$
\end{remark}

From the correctness of the algorithm, we obtain the following characterization for the existence of a feasible order.
\begin{theorem}\label{thm:gChar}
  For a digraph $D=(V,A)$ with a weight function $w$ on its arc set, the $(-\infty, g; \sum w)$-bounded ordering problem is polynomial-time solvable.
  There exists such an order if and only if $D$ has no induced subgraph $D' = (V', A')$ such that $\delta_w(v, V' \setminus \{v\}) > g(v)$ holds for each vertex $v \in V'$, where $\delta_w(v, V' \setminus \{v\})$ denotes the weighted outdegree of $v$ restricted to the arcs going out to the vertex set $V' \setminus \{v\}$.
  \FBOX
\end{theorem}

Note that the $(f,+\infty; \sum w)$\nobreakdash-bounded and the $(-\infty, g;\sum w)$\nobreakdash-bounded ordering problems are essentially the same in the sense that one can compute an $(f,+\infty; \sum w)$\nobreakdash-bounded order by reversing a $(-\infty, g;\sum w)$\nobreakdash-bounded order for $g \equiv \delta_w - f$.
Therefore, the case of only lower bounds can be solved with a similar algorithm, which fixes the vertices from left to right.
We state the corresponding theorem in the case when only lower bounds are given.
\begin{theorem}\label{thm:fChar}
  For a digraph $D = (V, A)$ with a weight function $w$ on its arc set, the $(f, +\infty; \sum w)$-bounded ordering problem is solvable in polynomial time.
  There exists such an order if and only if there is no induced subgraph $D' = (V', A')$ such that $\delta_w(v, V \setminus V') < f(v)$ holds for each vertex $v \in V'$, where $\delta_w(v, V \setminus V')$ denotes the weighted outdegree of $v$ restricted to the arcs going out to the vertex set $V \setminus V'$.
  \FBOX
\end{theorem}

Applying this theorem for unweighted digraphs and the lower bound $f(r)=0$, $f(v)=k$ for each $v\in V\setminus\{r\}$, where $r\in V$ is a fixed root vertex, we obtain the following corollary.
\begin{corollary}
  It can be decided in polynomial time whether a digraph contains $k$ arc-disjoint $r$\nobreakdash-in-ar\-bores\-cences whose union is acyclic.
  \FBOX
\end{corollary}
Note that the problem of finding arc-disjoint in-arborescences whose union is acyclic is only meaningful if they share the same root vertex.

\begin{remark}
  It is not difficult to show that the $(f,g; \sum w)$\nobreakdash-bounded ordering problem remains solvable even when both lower and upper bounds are given as long as, for each vertex, either a lower or an upper bound is specified.
  $\bullet$
\end{remark}

\subsubsection{Hardness results}\label{sec:Complexity}

In this section, we investigate the complexity of the $(f,g;\sum w)$\nobreakdash-bounded ordering problem.
In the previous section, we showed that the problem can be solved efficiently when only upper bounds are present.
This naturally raises the question whether a similar algorithm exists for more general cases or related problems.

First, we consider the case where we have upper bounds for all vertices except one for which both lower and upper bounds are given.
The NP-completeness of this variant was established for undirected graphs (which are essentially equivalent to symmetric digraphs) in~\cite{kiraly2018acyclic}.
However, the complexity of this problem remained an open question for simple graphs.
Now we prove that the problem remains NP-complete for simple digraphs.
\begin{theorem}\label{thm:oneStrictBound}
  The $(f,g)$\nobreakdash-bounded ordering problem for simple digraphs is NP-complete if only upper bounds are given for all vertices except for a single vertex $v$ for which $f(v)=g(v)$.
\end{theorem}
\begin{proof}
  The problem is clearly in NP.
  To prove that it is NP-complete, we reduce the independent set problem~\cite{karp1972reducibility} to it.
  Given an instance of the independent set problem, where we are given a graph $G=(V, E)$ and a parameter $k$, we construct a digraph $D=(V_D,A)$ and appropriate bounds such that the $(f,g)$\nobreakdash-bounded ordering problem on $D$ is solvable if and only if $G$ has an independent set of size $k$.
  The construction of the digraph $D$, illustrated by Figure~\ref{fig:oneStrictBound}, is as follows.
  Let the vertex set of $D$ consist of the vertices and edges of $G$, and an additional vertex~$s$.
  For each edge $e=uv\in E$, let $D$ contain an arc from $e\in V_D$ to $u\in V_D$ and an arc from $e\in V_D$ to $v\in V_D$.
  Moreover, for each vertex $v\in V$, let $D$ contain an arc from $s$ to $v\in V_D$.
  Let $D$ contain $n$ parallel arcs from $s$ to $e_1$ and two parallel arcs from every other vertex $e_i$ to the succeeding vertex $e_{i+1}$, where $e_1,\ldots,e_m$ are the vertices corresponding to the edges of $G$ in a fixed order, see the bottom row of Figure~\ref{fig:oneStrictBound}.
\FloatBarrier
  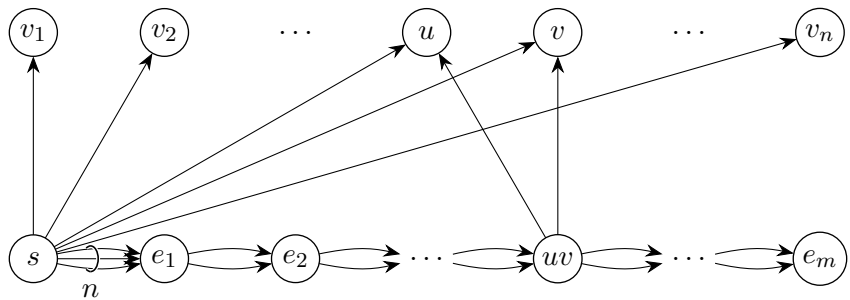
\begin{figure}[t]
    \centering
    \begin{tikzpicture}[yscale=1.5, xscale=1.15]
      \SetVertexMath
      \Vertex[x=-1.5, y=2,L=v_1]{v1}
      \Vertex[x=0, y=2,L=v_2]{v2}
      \draw (1.5,2) node () {$\ldots$};
      \Vertex[x=3, y=2, L=u]{u}
      \Vertex[x=4.5, y=2]{v}
      \draw (6,2) node () {$\ldots$};
      \Vertex[x=7.5, y=2, L=v_n]{vn}

      \Vertex[x=-1.5,y=0, L=s]{s}
      \Vertex[x=0,y=0, L=e_1]{e1}
      \Vertex[x=1.5,y=0, L=e_2]{e2}
      \draw (3,0) node () {$\ldots$};
      \Vertex[x=4.5,y=0]{uv}
      \draw (6,0) node () {$\ldots$};
      \Vertex[x=7.5,y=0, L=e_m]{em}

      \draw[\niceArrow] (s)--(v1);
      \draw[\niceArrow] (s)--(v2);
      \draw[\niceArrow] (s)--(u);
      \draw[\niceArrow] (s)--(v);
      \draw[\niceArrow] (s)--(vn);
      \draw[\niceArrow] (uv)--(u);
      \draw[\niceArrow] (uv)--(v);

      \draw[\niceArrow] (s)--(e1);
      \draw[\niceArrow] (s) to [out=10,in=170,looseness=1] (e1);
      \draw[\niceArrow] (s) to [out=-10,in=-170,looseness=1] (e1);
      \path[circleAroundEdges={.1}{.17}{.4}] (s)--(e1) node [below=.2 cm, pos=2/5] {$n$};
      \draw[\niceArrow] (e1) to [out=10,in=170,looseness=1] (e2);
      \draw[\niceArrow] (e1) to [out=-10,in=-170,looseness=1] (e2);
      \draw[\niceArrow] (e2) to [out=10,in=170,looseness=1] (2.7,0.04);
      \draw[\niceArrow] (e2) to [out=-10,in=-170,looseness=1] (2.7,-0.04);
      \draw[\niceArrow] (3.3,0.04) to [out=10,in=170,looseness=1] (uv);
      \draw[\niceArrow] (3.3,-0.04) to [out=-10,in=-170,looseness=1] (uv);
      \draw[\niceArrow] (uv) to [out=10,in=170,looseness=1] (5.7,0.04);
      \draw[\niceArrow] (uv) to [out=-10,in=-170,looseness=1] (5.7,-0.04);

      \draw[\niceArrow] (6.3,0.04) to [out=10,in=170,looseness=1] (em);
      \draw[\niceArrow] (6.3,-0.04) to [out=-10,in=-170,looseness=1] (em);

    \end{tikzpicture}
    \caption{The digraph $D$ constructed during the reduction from the independent set problem.}\label{fig:oneStrictBound}
  \end{figure}

  Let $f(s)=g(s)=k$ and $f(v)=-\infty$, $g(v)=1$ for each vertex $v\in V_D \setminus \{s\}$.
  We prove that the $(f,g)$\nobreakdash-bounded ordering problem is solvable for $D$ if and only if $G$ has an independent set of size $k$.

  Suppose that $G$ has an independent set of size $k$, and consider the following ordering of the vertices of $D$.
  First, list the vertices of the independent set in arbitrary order, then the vertices in the bottom row of Figure~\ref{fig:oneStrictBound} in the order $s,e_1,\ldots,e_m$, and put the remaining vertices in arbitrary order to the end.
  Then $\delta^{\ell}(s)=k$, $\delta^{\ell}(v)=0$ for each vertex $v\in V_D\cap V$, and each vertex $e\in V_D\cap E$ can only have left out-neighbors from the vertices of the independent set, therefore, $\delta^{\ell}(e) \leq 1$.
  So the resulting order is a feasible solution to the $(f,g)$\nobreakdash-bounded ordering problem defined on $D$.

  Conversely, suppose that there exists a feasible order $\sigma$ for the $(f,g)$\nobreakdash-bounded ordering problem.
  Because of the parallel arcs, the vertices $s,e_1,\dots,e_m$ in the bottom row of Figure~\ref{fig:oneStrictBound} must be in the given order.
  This implies that $s$ precedes all vertices $e\in V_D\cap E$ corresponding to the edges of $G$.
  The vertex $s$ has bounds $f(s)=g(s)=k$, therefore, there must be exactly $k$ vertices from $V_D\cap V$ before $s$ in $\sigma$.
  These vertices of $D$ correspond to vertices of $G$, and they must be independent in $G$, because the upper bound $g(e)=1$ for any edge $e\in E\cap V_D$ ensures that no two adjacent vertices may precede $e$.
  This implies that the first $k$ vertices in $\sigma$ form an independent set in $G$.

  The parallel arcs in the bottom row of Figure~\ref{fig:oneStrictBound} are only used to ensure the order of the vertices $s,e_1,\ldots,e_m$.
  We can also provide this by splitting each parallel arc $e$ with a new vertex $p_e$ with upper bound $g(p_e)=0$.
  This implies that the $(f,g)$\nobreakdash-bounded ordering problem remains NP-complete for simple digraphs.
  \hfill\end{proof}

In another natural modification, we no longer require the arc weights to be non-negative.
We show that this version of the problem is also NP-complete, using a similar reduction as in the previous proof.
\begin{theorem}
  The $(-\infty, g; \sum w)$\nobreakdash-bounded ordering problem for simple digraphs is NP-complete when negative arc weights are allowed.
\end{theorem}
\begin{proof}
  Given an instance of the independent set problem with a graph $G=(V,E)$ and a parameter $k$, we construct the digraph $D=(V_D,A)$ as described in the proof of Theorem~\ref{thm:oneStrictBound}, see Figure~\ref{fig:oneStrictBound}, and define the arc weights in such a way that the $(-\infty, g; \sum w)$\nobreakdash-bounded ordering problem for $D$ is solvable if and only if $G$ has an independent set of size $k$.
  The arc weights in $D$ are defined as $w(sv)=-1$ for each $v\in V$, and $w(e)=1$ for all other arcs.
  Let the upper bounds be $g(s)=-k$ and $g(v)=1$ for each vertex $v\in V_D \setminus \{s\}$.
  Similarly to the proof of Theorem~\ref{thm:oneStrictBound}, we can argue that $G$ has an independent set of size $k$ if and only if the $(-\infty, g; \sum w)$\nobreakdash-bounded ordering problem on $D$ is solvable.
  Moreover, this statement holds even if each parallel arc in the bottom row of Figure~\ref{fig:oneStrictBound} is divided by a new vertex with upper bound $0$.
\end{proof}

Another line of questions concerns the complexity of the $(f,g)$\nobreakdash-bounded ordering problem with special bound functions.
One such case is when the lower and upper bounds are equal for each vertex, in other words, there is an exact prescription for the left-outdegrees of the vertices.
In this section, we prove this problem to be NP-complete, however, it is known to be solvable in polynomial time for symmetric digraphs~\cite{kiraly2018acyclic}.
The other extreme case is when there is a significant difference between the lower and upper bounds on each vertex.
Suppose we are given a first vertex $s$ and a last vertex $t$ with bounds $f(s) = g(s) = 0$ and $f(t) = g(t) = \delta(t)$, respectively.
For every other vertex $v \notin \{s,t\}$, let the bounds be $f(v) = a$ and $g(v) = \delta(v) - b$, where $a$ and $b$ are given non-negative integers.
This problem is equivalent to ordering the vertices of a digraph such that each vertex has at least $a$ outgoing arcs to preceding vertices and at least $b$ outgoing arcs to succeeding vertices, except for $s$ and $t$.
If the parameters are $a=b=1$, then the problem is the so-called $s$\nobreakdash-$t$ numbering problem for digraphs, which is known to be polynomial-time solvable~\cite{cheriyan1994directeds}.
If $a=b=2$, then the problem becomes NP-complete even for symmetric digraphs~\cite{kiraly2018acyclic}.
This immediately implies that the problem is also NP-complete for any parameters where $a \geq 2$ and $b \geq 2$.
In what follows, we prove that the problem is also hard in the only remaining case $a=1$ and $b=2$.
\begin{theorem}\label{thm:Left1Right2}
  The $(f,g)$\nobreakdash-bounded ordering problem is NP-complete with bounds $f(v) = 1$, $g(v) = \delta(v) - 2$ for each vertex $v$, except for the first and last vertices.
  The problem remains NP-complete even when all outdegrees are at most $3$.
\end{theorem}
\begin{proof}
  The proof is by reduction from the NP-complete 3\nobreakdash-XSAT\nobreakdash-3 problem~\cite{porschen2014xsat}.
  In the 3\nobreakdash-XSAT\nobreakdash-3 problem, we are given a conjunctive normal form (CNF) formula in which each clause contains exactly three literals and each variable appears in exactly three clauses.
  The goal is to decide whether the formula can be satisfied such that exactly one literal is true in each clause.
  Let $x_1, \ldots, x_n$ denote the variables and $c_1, \ldots, c_n$ the clauses.
  \begin{figure}[t]
    \centering
    \begin{tikzpicture}[xscale=0.68]
      \Vertex[x=11.25, y=0]{t}

      \SetVertexNoLabel
      \grEmptyPath[form=1,x=-3.75,y=4.5,RA=6, rotation=0, prefix=x]{3}
      \AssignVertexLabel{x}{$x_1$, $x_2$, $x_3$}
      \grEmptyPath[form=1,x=-0.75,y=4.5,RA=6, rotation=0, prefix=nex]{3}
      \AssignVertexLabel{nex}{$\overline{x}_1$, $\overline{x}_2$, $\overline{x}_3$}

      \grEmptyPath[form=1,x=-2.25,y=0,RA=1.5, rotation=0, prefix=le]{3}
      \AssignVertexLabel{le}{$v_1$, $v_2$, $v_3$}
      \grEmptyPath[form=1,x=2.25,y=0,RA=1.5, rotation=0, prefix=c]{3}
      \AssignVertexLabel{c}{$c_1$, $c_2$, $c_3$}
      \grEmptyPath[form=1,x=6.75,y=0,RA=1.5, rotation=0, prefix=lu]{3}
      \AssignVertexLabel{lu}{$v'_1$, $v'_2$, $v'_3$}

      \grEmptyPath[form=1,x=-1.5,y=-1,RA=1.5, rotation=0, prefix=v]{8}
      \AssignVertexLabel{v}{$p_1$, $p_2$, $p_3$, $p_4$, $p_5$, $p_6$, $p_7$, $p_8$}

      \grEmptyPath[form=1,x=-6.75,y=0,RA=1.5, rotation=0, prefix=s]{3}
      \AssignVertexLabel{s}{$s_1$, $s_2$, $s_3$}

      \grEmptyPath[form=1,x=11.25,y=0,RA=1.5, rotation=0, prefix=t]{3}
      \AssignVertexLabel{t}{$t_1$, $t_2$, $t_3$}

      \draw[\niceArrow] (le0)--(x0)[color=lightgray];
      \draw[\niceArrow] (le0)--(nex0)[color=lightgray];
      \draw[\niceArrow] (lu0)--(x0)[color=lightgray];
      \draw[\niceArrow] (lu0)--(nex0)[color=lightgray];
      \draw[\niceArrow] (le1)--(x1)[color=lightgray];
      \draw[\niceArrow] (le1)--(nex1)[color=lightgray];
      \draw[\niceArrow] (lu1)--(x1)[color=lightgray];
      \draw[\niceArrow] (lu1)--(nex1)[color=lightgray];
      \draw[\niceArrow] (le2)--(x2)[color=lightgray];
      \draw[\niceArrow] (le2)--(nex2)[color=lightgray];
      \draw[\niceArrow] (lu2)--(x2)[color=lightgray];
      \draw[\niceArrow] (lu2)--(nex2)[color=lightgray];

      \draw[\niceArrow] (v0)--(le0);
      \draw[\niceArrow] (v0) to [out=75,in=225,looseness=1] (le1);
      \draw[\niceArrow] (v0) to [out=45,in=255,looseness=1] (le1);

      \draw[\niceArrow] (v1)--(le1);
      \draw[\niceArrow] (v1) to [out=75,in=225,looseness=1] (le2);
      \draw[\niceArrow] (v1) to [out=45,in=255,looseness=1] (le2);

      \draw[\niceArrow] (v2)--(le2);
      \draw[\niceArrow] (v2) to [out=75,in=225,looseness=1] (c0);
      \draw[\niceArrow] (v2) to [out=45,in=255,looseness=1] (c0);

      \draw[\niceArrow] (v3)--(c0);
      \draw[\niceArrow] (v3) to [out=75,in=225,looseness=1] (c1);
      \draw[\niceArrow] (v3) to [out=45,in=255,looseness=1] (c1);

      \draw[\niceArrow] (v4)--(c1);
      \draw[\niceArrow] (v4) to [out=75,in=225,looseness=1] (c2);
      \draw[\niceArrow] (v4) to [out=45,in=255,looseness=1] (c2);

      \draw[\niceArrow] (v5)--(c2);
      \draw[\niceArrow] (v5) to [out=75,in=225,looseness=1] (lu0);
      \draw[\niceArrow] (v5) to [out=45,in=255,looseness=1] (lu0);

      \draw[\niceArrow] (v6)--(lu0);
      \draw[\niceArrow] (v6) to [out=75,in=225,looseness=1] (lu1);
      \draw[\niceArrow] (v6) to [out=45,in=255,looseness=1] (lu1);

      \draw[\niceArrow] (v7)--(lu1);
      \draw[\niceArrow] (v7) to [out=75,in=225,looseness=1] (lu2);
      \draw[\niceArrow] (v7) to [out=45,in=255,looseness=1] (lu2);

      \draw[\niceArrow] (x0)--(s0)[color=lightgray];
      \draw[\niceArrow] (nex0)--(s0)[color=lightgray];
      \draw[\niceArrow] (x1)--(s1)[color=lightgray];
      \draw[\niceArrow] (nex1)--(s1)[color=lightgray];
      \draw[\niceArrow] (x2)--(s2)[color=lightgray];
      \draw[\niceArrow] (nex2)--(s2)[color=lightgray];

      \draw[\niceArrow] (x0) to [out=-23,in=168,looseness=0.2] (t0) [color=lightgray];
      \draw[\niceArrow] (x0) to [out=-13,in=158,looseness=0.2] (t0) [color=lightgray];
      \draw[\niceArrow] (nex0) to [out=-25,in=163,looseness=0.3] (t0) [color=lightgray];
      \draw[\niceArrow] (nex0) to [out=-15,in=153,looseness=0.3] (t0) [color=lightgray];

      \draw[\niceArrow] (x1) to [out=-30,in=160,looseness=0.4] (t1) [color=lightgray];
      \draw[\niceArrow] (x1) to [out=-20,in=150,looseness=0.4] (t1) [color=lightgray];
      \draw[\niceArrow] (nex1) to [out=-35,in=155,looseness=0.5] (t1) [color=lightgray];
      \draw[\niceArrow] (nex1) to [out=-25,in=145,looseness=0.5] (t1) [color=lightgray];

      \draw[\niceArrow] (x2) to [out=-40,in=150,looseness=0.6] (t2) [color=lightgray];
      \draw[\niceArrow] (x2) to [out=-30,in=140,looseness=0.6] (t2) [color=lightgray];
      \draw[\niceArrow] (nex2) to [out=-60,in=130,looseness=0.7] (t2) [color=lightgray];
      \draw[\niceArrow] (nex2) to [out=-50,in=120,looseness=0.7] (t2) [color=lightgray];

      \draw[\niceArrow] (s1) to [out=165,in=15,looseness=1] (s0);
      \draw[\niceArrow] (s0) to [out=-30,in=-150,looseness=1] (s1);
      \draw[\niceArrow] (s0) to [out=-15,in=-165,looseness=1] (s1);

      \draw[\niceArrow] (s2) to [out=165,in=15,looseness=1] (s1);
      \draw[\niceArrow] (s1) to [out=-30,in=-150,looseness=1] (s2);
      \draw[\niceArrow] (s1) to [out=-15,in=-165,looseness=1] (s2);

      \draw[\niceArrow] (s2) to [out=-30,in=-150,looseness=1] (le0);
      \draw[\niceArrow] (s2) to [out=-15,in=-165,looseness=1] (le0);

      \draw[\niceArrow] (t0) to [out=165,in=15,looseness=1] (lu2);

      \draw[\niceArrow] (t1) to [out=165,in=15,looseness=1] (t0);
      \draw[\niceArrow] (t0) to [out=-30,in=-150,looseness=1] (t1);
      \draw[\niceArrow] (t0) to [out=-15,in=-165,looseness=1] (t1);

      \draw[\niceArrow] (t2) to [out=165,in=15,looseness=1] (t1);
      \draw[\niceArrow] (t1) to [out=-30,in=-150,looseness=1] (t2);
      \draw[\niceArrow] (t1) to [out=-15,in=-165,looseness=1] (t2);

      \draw[\niceArrow] (le0)--(le1);
      \draw[\niceArrow] (le1)--(le2);
      \draw[\niceArrow] (le2)--(c0);
      \draw[\niceArrow] (lu0)--(lu1);
      \draw[\niceArrow] (lu1)--(lu2);
      \draw[\niceArrow] (lu2) to [out=-15,in=-165,looseness=1] (t0);

      \draw[\niceArrow] (c0)--(x0);
      \draw[\niceArrow] (c0)--(x1);
      \draw[\niceArrow] (c0)--(x2);
      \draw[\niceArrow] (c1)--(nex0);
      \draw[\niceArrow] (c1)--(nex1);
      \draw[\niceArrow] (c1)--(x2);
      \draw[\niceArrow] (c2)--(nex0);
      \draw[\niceArrow] (c2)--(x1);
      \draw[\niceArrow] (c2)--(nex2);

    \end{tikzpicture}
    \caption{The digraph constructed in the proof of Theorem~\ref{thm:Left1Right2} for the CNF formula $(x_1 \lor x_2 \lor x_3) \land (\overline{x}_1 \lor \overline{x}_2 \lor x_3) \land (\overline{x}_1 \lor x_2 \lor \overline{x}_3)$.
    }\label{fig:Left1Right2}
  \end{figure}

  We construct an instance of the $(f,g)$\nobreakdash-bounded ordering problem as follows.
  Let $D$ be the digraph containing the vertices $x_i$ and $\overline{x}_i$ for each literal, two vertices $v_i$ and $v'_i$ for the $i^{\text{th}}$ variable for $i\in \{1,\ldots,n\}$, a vertex $c_j$ corresponding to the $j^{\text{th}}$ clause for $j\in\{1,\ldots,n\}$.
  We add two further vertices denoted by $s_i$ and $t_i$ for $i\in\{1,\ldots,n\}$, and an arc from $x_i$ and $\overline{x}_i$ to $s_i$ and a pair of parallel arcs to $t_i$.
  For each $i\in\{1,\ldots,n\}$, the vertices $v_i$ and $v'_i$ have arcs going to the vertices $x_i$ and $\overline{x}_i$.
  Each clause vertex $c_j$ has three arcs going out to the literals $x_i$ or $\overline{x}_i$ contained in $c_j$.
  Consider the vertices $s_1, \ldots, s_n, v_1,\ldots, v_n, c_1, \ldots, c_n, v'_1, \ldots, v'_n, t_1, \ldots, t_n$ in this order.
  Let $D$ contain an arc from the vertices $v_i$ and $v'_i$ to the vertex immediately to their right, and from each vertex $s_i$ and $t_i$ two parallel arcs going out to the next vertex on their right and one arc going out to the vertex on their left.
  Moreover, let $D$ contain an additional vertex between every two adjacent vertices of the sequence $v_1,\ldots, v_n, c_1, \ldots, c_n, v'_1, \ldots, v'_n$, and let $p_1,\ldots,p_{3n-1}$ denote these newly added vertices.
  For each $k\in\{1,\ldots,3n-1\}$, let $p_k$ have two parallel arcs going out to the succeeding vertex and one arc going out to the preceding vertex.
  Figure~\ref{fig:Left1Right2} gives an example for the construction.
  Let the bounds be $f(s_1)=g(s_1)=0$ and $f(t_n)=g(t_n)=1$, and for each vertex $v\notin \{s_1,t_n\}$, $f(v)=1$ and $g(v)=\delta(v)-2$.
  We prove that the 3\nobreakdash-XSAT\nobreakdash-3 problem is satisfiable if and only if the corresponding $(f,g)$\nobreakdash-bounded ordering problem is solvable.

  Suppose that the instance of the 3\nobreakdash-XSAT\nobreakdash-3 problem is satisfiable.
  Consider the following order of the vertices of $D$: First, list the vertices $s_1, \ldots, s_n, v_1,\ldots, v_n, c_1, \ldots, c_n, v'_1, \ldots, v'_n,$ $t_1, \ldots, t_n$ in this order, and put each additional vertex $p_k$ between its two neighbors.
  Then put the vertices of the true literals right after the vertex $s_n$ and the vertices of the false literals right before the vertex $t_1$.
  By the construction, it is easy to see that the resulting order is a feasible solution to the $(f,g)$\nobreakdash-bounded ordering problem.

  Conversely, suppose that there exists a feasible order $\sigma$ to the $(f,g)$\nobreakdash-bounded ordering problem.
  Notice that the additional vertices $p_1,\ldots,p_{3n-1}$ ensure that the vertices $v_1,\ldots, v_n,$ $c_1, \ldots, c_n, v'_1, \ldots, v'_n$ appear in this order.
  Furthermore, the vertices $s_1, \ldots, s_n$ must precede this sequence,
  and the vertices $t_1,\ldots, t_n$ must succeed it because of their outgoing parallel arcs.
  Therefore, the vertices
  $s_1, \ldots, s_n, v_1,\ldots, v_n,$ $c_1, \ldots, c_n,v'_1,\ldots, v'_n, t_1, \ldots, t_n$ appear in this order.
  For each $v_i$ and $v'_i$, one of the outgoing arcs must point to the left and the other two must point to the right, thus placing one of $x_i$ or $\overline{x}_i$ before $v_i$ and the other after $v'_i$.
  This means that, for each variable, one of the vertices $x_i$ and $\overline{x}_i$ precedes the vertex $v_i$ and the other one succeeds the vertex $v'_i$.
  By the fixed order of the vertices $v_1,\ldots, v_n, c_1, \ldots, c_n,$ $v'_1,\ldots, v'_n$, this implies that, for each variable, one of the two literals $x_i$ and $\overline{x}_i$ must be placed before $c_1,\ldots, c_n$ and the other one after them.
  Set the variable $x_i$ to true if the literal $x_i$ precedes the vertices $c_1, \ldots, c_n$, and to false otherwise, hence exactly the literals preceding the vertices $c_1,\ldots,c_n$ are true.
  This is a solution to the instance of the 3\nobreakdash-XSAT\nobreakdash-3 problem, because each vertex $c_j$ has three arcs going out to the vertices corresponding to the literals in $c_j$, and exactly one of these precedes the vertex $c_j$ by the bounds $f$ and $g$ --- and hence it also precedes $c_1,\ldots,c_n$.
  Therefore, exactly one literal is true in each clause.
\end{proof}

For the digraph $D$, the $(f,g)$\nobreakdash-bounded ordering problem defined in the proof is solvable if and only if the same problem with bounds $f(s_1) = g(s_1) = 0$ and $f(v) = g(v) = 1$ for each vertex $v \neq s_1$ is solvable.
Therefore, the proof also shows that the case with prescribed left-outdegrees (when $f \equiv g$) is NP-complete.
In contrast, the problem with prescribed left-outdegrees is polynomial-time solvable for symmetric digraphs (i.e., undirected graphs)~\cite{kiraly2018acyclic}.
\begin{corollary}\label{cor:strict}
  The $(f,g)$\nobreakdash-bounded ordering problem with $f \equiv g$ is NP-complete even when the bounds $f \equiv g$ are $0$ for one vertex and $1$ for all other vertices.
  \FBOX
\end{corollary}

In summary, the problem is NP-complete for all parameters $a \geq 1$ and $b \geq 2$, covering essentially all cases with stricter bounds than those in the $s$\nobreakdash-$t$ numbering problem.

In the next two sections, we investigate modified versions of the $(f,g)$\nobreakdash-bounded ordering problem.

\subsubsection{The \texorpdfstring{$d$-distance $(-\infty, g)$-bounded}{d-distance (-infty, g)-bounded} ordering problem}\label{sec:dDist}
This section introduces a simple modification of the $(-\infty, g)$\nobreakdash-bounded ordering problem, where the upper bound $g(v)$ only applies to the set of the (at most) $d$ vertices directly preceding the vertex $v$, i.e., for $i \in \{1, \dots, |V|\}$, the bound for the $i^{\text{th}}$ vertex applies to the set of the preceding $\min\{d, i-1\}$ vertices.
We refer to this problem as the $d$\nobreakdash-distance $(-\infty, g)$\nobreakdash-bounded ordering problem, inspired by the $d$\nobreakdash-distance matching problem~\cite{madarasi2020distance,madarasi2021matchings,madarasi2024matchings}.

In the special case when $d = |V|$, this problem is the usual $(-\infty, g)$\nobreakdash-bounded ordering problem, which was shown to be polynomial-time solvable in Section~\ref{sec:OnlyUpperBounds}.
Now we investigate the complexity of the $d$\nobreakdash-distance $(-\infty, g)$\nobreakdash-bounded ordering problem for other values of $d$.
\begin{theorem}
  For a digraph $D=(V,A)$ and a parameter $d=|V|-k$ for a fixed integer $k$, the $d$\nobreakdash-distance $(-\infty, g)$\nobreakdash-bounded ordering problem is polynomial-time solvable.
\end{theorem}
\begin{proof}
  We fix the first $k$ and the last $k$ vertices of the order in every possible way.
  For each fixed vertex, we can determine the set of its (at most) $(|V| - k)$ preceding vertices.
  We can then check whether these vertices violate the upper bounds.
  If they do not violate the bounds, then we aim to find a feasible order for the remaining middle $(|V| - 2k)$ vertices, or prove that the order of the already fixed vertices cannot be completed.
  For the middle $(|V| - 2k)$ vertices, define a new $(-\infty, g)$\nobreakdash-bounded ordering problem: For each vertex $v$ in the middle set, the set of its preceding (at most) $(|V| - k)$ vertices now includes the first $k$ fixed vertices.
  We remove these vertices and reduce the upper bound $g(v)$ by the number of arcs going to the fixed first $k$ vertices.
  It is easy to see that the order of the already fixed vertices can be completed into a feasible order if and only if the bounded ordering problem defined on the middle $(|V| - 2k)$ vertices is solvable.

  This method is polynomial because the first $k$ and last $k$ vertices can be fixed in $\frac{n!}{(n-2k)!}$ different ways, so the polynomial-time algorithm for the $(-\infty, g)$\nobreakdash-bounded ordering problem is called at most that many times for the middle $(|V| - 2k)$ vertices.
\end{proof}

A natural question arises: Does the problem remain solvable for smaller values of $d$?
We show that the problem becomes NP-complete for appropriately small $d$, including the case when $d$ is a constant.
\begin{theorem}
  For a digraph $D = (V, A)$, the $d$\nobreakdash-distance $(-\infty, g)$\nobreakdash-bounded ordering problem is NP-complete if $|V| = (d+1)(l+1) - 2$ for an integer $l = \Omega(|V|^c)$ with some constant $c > 0$.
\end{theorem}
\begin{proof}
  We first prove the hardness for the case $d = 1$ by reduction from the Hamiltonian path problem, which is known to be NP-complete~\cite{karp1972reducibility}.
  Then, we extend the proof for larger values of $d$ by a reduction from the $1$\nobreakdash-distance $(-\infty, g)$\nobreakdash-bounded ordering problem.

  In the case of the $1$\nobreakdash-distance $(-\infty, g)$\nobreakdash-bounded ordering problem with $g \equiv 0$, the goal is to find an ordering of the vertices such that no vertex has an arc going out to the vertex directly preceding it.
  Let $G$ be the graph on $2l$ vertices for which we want to solve the Hamiltonian path problem.
  Consider the complement of $G$ and direct each edge in both directions.
  Let $D$ denote the resulting digraph.
  We show that there exists a Hamiltonian path in $G$ if and only if the $1$\nobreakdash-distance $(-\infty, g)$\nobreakdash-bounded ordering problem with $g \equiv 0$ is solvable for $D$.

  First, suppose that the graph $G$ contains a Hamiltonian path and order the vertices according to the Hamiltonian path.
  Then, no two adjacent vertices have an arc between them in $D$, so this order is a feasible solution to the $1$\nobreakdash-distance $(-\infty,g)$\nobreakdash-bounded ordering problem on $D$.
  Conversely, if there exists a solution to the $1$\nobreakdash-distance $(f,g)$\nobreakdash-bounded ordering problem on $D$, then there is no arc from any vertex to its preceding vertex.
  Because of the construction of $D$, this means that there is an edge in $G$ between any two adjacent vertices, so the vertices form a Hamiltonian path in this order.

  Now, for larger values of $d$, we reduce the case $d = 1$ to the cases where $d \geq 2$.
  Let $D'$ be the digraph on $2l$ vertices for which we want to solve the $1$\nobreakdash-distance $(-\infty, g)$\nobreakdash-bounded ordering problem with $g \equiv 0$, which is NP-complete as shown above.
  We extend $D'$ by adding $(d - 1)$ complete symmetric digraphs on $(l + 1)$ vertices.
  Let $K_1, \dots, K_{d-1}$ denote these newly added complete digraphs.
  The vertices in $D'$ are called the original vertices, and the vertices in $K_1, \dots, K_{d-1}$ are called auxiliary vertices.
  We show that, with the upper bound $g \equiv 0$, the $1$\nobreakdash-distance $(-\infty, g)$\nobreakdash-bounded ordering problem is solvable for $D'$ if and only if the $d$\nobreakdash-distance $(-\infty, g)$\nobreakdash-bounded ordering problem is solvable for $D$.

  First, consider a feasible order $\sigma'$ for the problem defined on $D'$.
  This means that there is no arc from any vertex to its preceding vertex in $\sigma'$.
  Consider the following order $\sigma$ of the vertices in $D$: List the vertices of $D'$ in pairs according to the order $\sigma'$.
  Before the first vertex, after the last vertex, and between every two adjacent blocks of two vertices, put one auxiliary vertex from each complete digraph $K_j$ in the order of the indices of the complete digraphs.
  This order $\sigma$ is a solution to the problem defined on $D$, because the auxiliary vertices are only adjacent with vertices from the same complete digraph, and these vertices are at least $(d + 1)$ distance apart from each other.
  For each original vertex, the set of its preceding (at most) $d$ vertices contains $(d - 1)$ auxiliary vertices and (at most) one original vertex.
  Therefore, if there are no arcs going out from the vertices to their preceding vertex in $\sigma'$, then there are no arcs going out from the vertices to their preceding (at most) $d$ vertices in $\sigma$.

  Conversely, suppose there exists a feasible order $\sigma$ to the problem defined on $D$.
  We show that in this order, there are two original vertices and $(d - 1)$ auxiliary vertices alternately such that the first and last $(d - 1)$ vertices are auxiliary vertices.
  Suppose indirectly that the order does not follow this pattern.
  Then, for some $i\in \{1,\ldots,l\}$, at the positions of the $i^{\text{th}}$ pair of original vertices, there is at least one auxiliary vertex from the complete digraph $K_j$.
  Notice that there are at most $(i(d+1) - 1)$ vertices before this vertex and at most $((l + 1 - i) (d + 1) - 1)$ vertices after this vertex.
  Furthermore, each two auxiliary vertices from the same complete digraph $K_j$ must be at least $(d + 1)$ distance apart from each other.
  From these, we get the following upper bound for the number of the vertices in the complete digraph $K_j$:
  \[
    \left\lfloor\frac{i(d+1)-1}{d+1}\right\rfloor+1+\left\lfloor\frac{\left(l+1-i\right)(d+1)-1}{d+1}\right\rfloor<l+1.
  \]
  This is a contradiction, because each complete digraph has $(l + 1)$ vertices.
  So the order $\sigma$ follows the pattern which was described above.
  This means that for each vertex, the set of its preceding (at most) $d$ vertices contains $(d - 1)$ auxiliary vertices and (at most) one original vertex.
  By leaving out the auxiliary vertices from $\sigma$, there are no arcs from the vertices to their preceding vertex in $D'$.
  Therefore, the resulting order is a feasible solution to the $1$\nobreakdash-distance $(-\infty, g)$\nobreakdash-bounded ordering problem with bound $g\equiv 0$.
\end{proof}

\subsubsection{Lexicographical \texorpdfstring{$(-\infty, g)$-bounded}{(-infty, g)-bounded} orders}\label{sec:lexicographical}
This section discusses different types of lexicographically minimal or maximal $(-\infty, g)$\nobreakdash-bounded orders.
First, an order is \emph{decreasingly minimal} if the vector of the left-outdegrees ordered non-increasingly is lexicographically minimal.
This concept arises in applications such as rank aggregation (see Section~\ref{sec:introduction}).
In this case, the left-outdegree of a vertex $v$ measures how ``unfair'' the common ranking appears for the candidate $v$.
A decreasingly minimal order minimizes the ``disappointment'' of the most disappointed candidate, then that of the second most disappointed, and so on.
Unfortunately, the decreasingly minimal $(-\infty,g)$\nobreakdash-bounded ordering problem turns out to be NP-hard.
\begin{theorem}\label{thm:decMinBoundedOrder}
  It is NP-hard to find a decreasingly minimal $(-\infty, g)$\nobreakdash-bounded order for $g \equiv 1$.
\end{theorem}
\begin{proof}
  The proof consists of two parts.
  First, we show that the problem of finding a decreasingly minimal $(-\infty, g)$\nobreakdash-bounded order is equivalent to finding a minimum-size in-branching that covers all directed cycles.
  Then, we prove that the latter problem is APX-hard.
  By definition, an order is $(-\infty, g)$\nobreakdash-bounded for $g \equiv 1$ if and only if the arcs going to the left form an in-branching.
  Clearly, such an order is decreasingly minimal if and only if it minimizes the number of vertices with left-outdegree exactly one, in other words, it minimizes the size of the in-branching.

  We now show that, for a digraph $D = (V, A)$, finding an order such that the left-going arcs form a minimum-size in-branching is essentially the same as finding a minimum-size in-branching that covers all directed cycles.
  Suppose we have an order where the left-going arcs form an in-branching of size $k$.
  Since the right-going arcs form an acyclic subgraph, the in-branching must cover all directed cycles.
  For the other direction, suppose we have an in-branching $B \subseteq A$ of size $k$ that covers all directed cycles.
  Let $\sigma$ be the topological ordering of the subgraph $D \setminus B$.
  The left-going arcs in $\sigma$ form a subset of $B$, so $\sigma$ is an order in which the left-going arcs form an in-branching of size at most $k$.
  Thus, the problem of finding a decreasingly minimal $(-\infty, g)$\nobreakdash-bounded ordering is equivalent to finding a minimum-size in-branching that covers all directed cycles.

  Next, we show that finding a minimum-size in-branching that covers all directed cycles is APX-hard by an approximation-preserving reduction from the feedback arc set problem, which is known to be APX-hard~\cite{kann1992approximability}.
  Let $D' = (V', A')$ be a digraph, and construct a new digraph $D = (V, A)$ by splitting every arc $a' \in A'$ with a new vertex $v_{a'}$.

  We prove that there exists a feedback arc set of size at most $k$ in $D'$ if and only if there exists an in-branching of size at most $k$ covering all directed cycles in $D$.
  Suppose there exists a feedback arc set $F \subseteq A'$ of size $k$ in $D'$.
  Consider the set of arcs that are going out from the vertices $v_{a'}$ for some $a' \in F$.
  These arcs form an in-branching of size $k$ that covers all directed cycles in $D$.
  For the other direction, suppose there exists an in-branching $B \subseteq A$ in $D$ that covers all directed cycles.
  Let $F \subseteq A'$ be the set of arcs $a'$ for which there exists an arc in $B$ incident to the vertex $v_{a'}$.
  Then, $F$ is a feedback arc set in $D'$ of size at most $k$.

  Thus, solving the problem of finding a minimum-size in-branching that covers all directed cycles in $D$ can be used to approximate the minimum feedback arc set problem in $D'$.
  Since the feedback arc set problem is APX-hard, it follows that finding a minimum-size in-branching covering all directed cycles is also APX-hard.
  This completes the proof, as the latter problem is equivalent to finding a decreasingly minimal $(-\infty, g)$\nobreakdash-bounded ordering.
\end{proof}

Next, we define other types of lexicographical orderings:
An order of the vertices is lexicographically minimal (maximal) from the left if the vector of left-outdegrees of the vertices from left to right is lexicographically minimal (maximal).
Similarly, an order is lexicographically minimal (maximal) from the right if the vector of left-outdegrees in the given order (from right to left) is lexicographically minimal (maximal).

Note that if the vertices of a simple digraph have an order in which the left-outdegree of the $i^{\text{th}}$ vertex is $\ell$, then in the same order for the complement digraph, the left-outdegree of the $i^{\text{th}}$ vertex is $(i - \ell - 1)$.
This fact leads to the following corollary.
\begin{corollary}\label{cor:lexMinEquivLexMax}
  Let $D$ be a simple digraph, and let $\overline{D}$ denote its complement.
  Then a lexicographically minimal order from the left for $D$ is lexicographically maximal from the left for $\overline{D}$, and a lexicographically minimal order from the right for $D$ is lexicographically maximal from the right for $\overline{D}$.
  \FBOX
\end{corollary}

Now we prove that both problems are NP-hard.
\begin{theorem}\label{thm:leftLexMin}
  It is NP-hard to find a lexicographically minimal (resp.\ maximal) $(-\infty, +\infty)$-bounded order from the left, even for a simple undirected graph, i.e., a symmetric digraph.
\end{theorem}
\begin{proof}
  The proof for the lexicographically minimal case is by reduction from the independent set problem, which is NP-complete~\cite{karp1972reducibility}.
  Consider the left-degree vector of a lexicographically minimal order from the left for an undirected graph $G=(V,E)$.
  We prove that the number of zeros at the beginning of this vector is equal to the size of the maximum independent set in $G$.
  If the vector starts with $k$ zeros, then the first $k$ vertices in the order form an independent set in $G$.
  Conversely, if there are $k$ independent vertices in $G$, then we can arrange these $k$ vertices at the start of the order, ensuring that the left-degree vector begins with $k$ zeros, which in turn means that the lexicographically minimal order also starts with (at least) $k$ zeros.
  Thus, finding a lexicographically minimal order from the left is equivalent to solving the independent set problem.

  Finally, note that the hardness of finding a lexicographically maximal order immediately follows by the hardness of minimization and by Corollary~\ref{cor:lexMinEquivLexMax}.
\end{proof}

Next, we show that the analogous problem of finding a lexicographically minimal order from the right is also NP-complete.
\begin{theorem}\label{thm:rightLexMin}
  It is NP-hard to find a lexicographically minimal (resp.\ maximal) $(-\infty, +\infty)$-bounded order from the right, even for a simple undirected graph, i.e., a symmetric digraph.
\end{theorem}
\begin{proof}
  The proof for the lexicographically minimal case is by reduction from the maximum clique problem, which is NP-complete~\cite{karp1972reducibility}.
  Consider the left-degree vector of a lexicographically minimal order from the right for an undirected graph $G=(V,E)$.
  We show that the length of the maximal strictly increasing sequence at the end of this vector is equal to the size of the maximum clique in the graph induced by the vertices with minimum degree in $G$.
  This problem is equivalent to the maximum clique problem.
  First, observe that in every minimal order, the last vertex is a vertex with minimum degree.
  If the vector ends with a strictly increasing sequence of length $k$, then the second-to-last vertex is also a vertex with minimum degree, and it has an arc to the last vertex.
  Similarly, the third-to-last vertex is a vertex with minimum degree, and it has arcs to both the last and the second-to-last vertices.
  This pattern continues for each $i = 1,\dots,k$: the $i^\text{th}$ last vertex is a vertex with minimum degree and has an arc to each of the vertices that follow it in the order.
  This means that the last $k$ vertices in the order form a clique in $G$.

  Conversely, let $k$ denote the size of the maximum clique in the induced subgraph of $G$ formed by the vertices with minimum degree.
  Let us denote the minimum degree by $d_{\min}$.
  Consider an order in which the last $k$ vertices correspond to the vertices in a maximum clique.
  Clearly, the last $k$ left-degrees are $d_{\min} - k + 1, \dots, d_{\min}$.
  Notice that in any order, the left-degree of the $i^\text{th}$ last vertex is at least $(d_{\min} - i + 1)$.
  Therefore, the last $k$ values of any lexicographically minimal order must be $d_{\min} - k + 1, \dots, d_{\min}$.
  This implies that the left-degree vector of a lexicographically minimal order ends with a strictly increasing sequence of length $k$.

  Thus, we have shown that the length of the maximal strictly increasing sequence at the end of the left-degree vector in a lexicographically minimal order is exactly the size of the maximum clique in the subgraph induced by the vertices with minimum degree.
  Since the latter problem is NP-hard, it follows that finding a lexicographically minimal order from the right is also NP-hard.

  The hardness of finding a lexicographically maximal order immediately follows by the hardness of minimization and by Corollary~\ref{cor:lexMinEquivLexMax}.
\end{proof}

\subsection{Indegree- and outdegree-bounded ordering problems}\label{sec:simultaneousBounds}

This section investigates Problem~\ref{prob:degreeBounded}, as a straightforward generalization of the $(f,g)$-bounded ordering problem discussed in Section~\ref{sec:fgBounded}.
Here, we have simultaneous bounds for the left-outdegrees and right-indegrees.
Observe that a lower or upper bound on the right-outdegrees can be expressed as an upper or lower bound on the left-outdegrees, respectively, and vice-versa.
An analogous statement holds for the indegrees.
Therefore, any vertex-ordering problem with simultaneous bounds on the left- and right-outdegrees and left- and right-indegrees can be reduced to Problem~\ref{prob:degreeBounded}.
In this section, we provide a comprehensive complexity analysis for cases where (lower, upper, or exact) bounds are given for the left-outdegrees and right-indegrees.
See Table~\ref{tab:simultaneous} for a summary.

First consider the cases when bounds are only imposed for either the left-outdegrees or the right-indegrees.
The diagonal of Table~\ref{tab:simultaneous} corresponds to cases in which a single type of bound is given.
We have already proved the complexities of the problems with lower bounded, upper bounded, or prescribed left-outdegrees in  Theorem~\ref{thm:fChar}, Theorem~\ref{thm:gChar}, and Corollary~\ref{cor:strict}, respectively.
These results can be applied to the cases where only the right-indegrees are bounded, by reversing both the arcs of the digraph and the vertex orderings.
Furthermore, the problem with simultaneous lower and upper bounds for the left-outdegrees is equivalent to the problem with simultaneous lower and upper bounds for the right-indegrees, by reversing both the digraph and the vertex orderings.
Since the former problem is the $(f,g)$\nobreakdash-bounded ordering problem, which is NP-complete as shown in Corollary~\ref{cor:LowerAndUpper}, it follows that the latter problem is also NP-complete.

Next, we consider the problems with a prescription for the left-outdegrees or right-indegrees and simultaneously a lower or upper bound for the right-indegrees, or left-outdegrees, respectively.
If the left-outdegrees are prescribed, then the NP-completeness of the problem follows by Corollary~\ref{cor:strict} applied in the case when the lower or upper bound is $-\infty$ or $+\infty$, respectively.
If the prescription is given for the right-indegrees, then the problem is also NP-complete by considering the reversed orders for the reversed digraph.

Now we prove the complexities of the remaining cases.
\begin{theorem}\label{thm:outUpperInLower}
  For a digraph $D = (V, A)$ with an upper bound $g_{\delta}: V \to \Z_+\cup \{+\infty\}$ on the left-outdegrees and a lower bound $f_{\varrho}: V \to \Z_+\cup \{-\infty\}$ on the right-indegrees, a vertex order can be found in polynomial time such that $\delta^{\ell}(v) \leq g_{\delta}(v)$ and $\varrho^{r}(v) \geq f_{\varrho}(v)$ hold for each $v \in V$.
\end{theorem}
\begin{proof}
  Consider the following modification of Algorithm~\ref{alg:UpperBounds}.
  In Line~\ref{alg:line:filter}, change the definition of $V^*$ to $V^* \coloneqq \{v\in V': \delta(v, V' \setminus \{v\}) \leq g_{\delta}(v)$ and $\varrho(v, V\setminus V') \geq f_{\varrho}(v)\}$, where $V'$ denotes the set of non-fixed vertices at any given iteration.
  The modified algorithm proceeds by fixing the vertices from right to left as follows.
  In each iteration, $V^*$ denotes the subset of $V'$ which satisfies both bounds.
  If $V^*$ is non-empty, then we pick an arbitrary vertex from $V^*$, place it at the last free position in the order, and remove it from $V'$.
  If $V^*$ is empty, then the algorithm concludes that no feasible order exists.
  We show that the modified algorithm correctly solves the problem.
  The proof follows a structure similar to the proof of Theorem~\ref{thm:upperBounds}.
  By the definition of $V^*$, any vertex fixed by the algorithm satisfies both the upper bound on the left-outdegree and the lower bound on the right-indegree.
  Thus, if $V^*$ is non-empty at each iteration of the algorithm, then a feasible order is found.
  Otherwise, $V^*$ is empty at some point in the algorithm.
  Let $\sigma$ be any arbitrary order of $V$, and let $u$ be the last vertex in $V'$ according to this order.
  Since $V^*$ is empty, it follows that $\delta^{\ell}(u) \geq \delta(u, V' \setminus \{u\})> g_{\delta}(u)$ or $\varrho^{r}(u) \leq \varrho(u, V \setminus V')< f_{\varrho}(u)$.
  This means that $u$ violates at least one of its degree bounds.
  Therefore, the order $\sigma$ is not feasible.
  Consequently, if $V^*$ is empty and there are still non-fixed vertices in $V'$, then no feasible order exists.
  Thus, the algorithm correctly finds a feasible vertex order or concludes that no such order exists, and the time complexity is polynomial.
\end{proof}

The next theorem can be proved analogously.
The only difference is that the algorithm fixes the vertices from left to right.
\begin{theorem}\label{thm:outLowerInUpper}
  For a digraph $D = (V, A)$ with a lower bound $f_{\delta}: V \to \Z_+\cup \{-\infty\}$ on the left-outdegrees  and an upper bound $g_{\varrho}: V \to \Z_+\cup \{+\infty\}$ on the right-indegrees, a vertex order can be found in polynomial time such that $\delta^{\ell}(v) \geq f_{\delta}(v)$ and $\varrho^{r}(v) \leq g_{\varrho}(v)$ hold for each $v\in V$.
  \FBOX
\end{theorem}

Corollary~\ref{cor:LowerAndUpper} immediately implies the following, because the left-indegrees and the left-out\-de\-grees are the same for symmetric digraphs, and a lower bound for the left-indegrees is equivalent to an upper bound for the right-indegrees.
\begin{theorem}\label{thm:outUpperInUpper}
  For a digraph $D = (V, A)$ with upper bounds $g_{\delta}: V \to \Z_+\cup \{+\infty\}$ and $g_{\varrho}: V \to \Z_+\cup \{+\infty\}$ on the left-outdegrees and right-indegrees, respectively, it is NP-complete to decide whether a vertex order exists such that $\delta^{\ell}(v) \leq g_{\delta}(v)$ and $\varrho^{r}(v) \leq g_{\varrho}(v)$ hold for each~$v\in V$.
  \FBOX
\end{theorem}

The problem with upper bounds on the left-outdegrees and right-indegrees can be transformed into a problem with lower bounds on the left-outdegrees and right-indegrees by considering the reverse orders.
This implies that the latter problem is also NP-complete.
\begin{corollary}\label{cor:outLowerInLower}
  For a digraph $D = (V, A)$ with lower bounds $f_{\delta}: V \to \Z_+\cup \{-\infty\}$ and $f_{\varrho}: V \to \Z_+\cup \{-\infty\}$ on the left-outdegrees and right-indegrees, respectively, it is NP-complete to decide whether a vertex order exists such that $\delta^{\ell}(v) \geq f_{\delta}(v)$ and $\varrho^{r}(v) \geq f_{\varrho}(v)$ hold for each~$v\in V$.
  \FBOX
\end{corollary}

Now, let us consider the case when both the left-outdegree and the right-indegree are prescribed exactly for each vertex.
Surprisingly, this strongly restrictive version of the problem turns out to be polynomial-time solvable, in contrast to the problems where either only the left-outdegrees or only the right-indegrees are prescribed.
\begin{theorem}\label{thm:outStrictInStrict}
  For a digraph $D = (V, A)$ with exact prescriptions $m_{\delta}: V \to \Z_+$ and $m_{\varrho}: V \to \Z_+$ for the left-outdegrees and right-indegrees, respectively, a vertex order can be found in polynomial time such that $\delta^{\ell}(v) = m_{\delta}(v)$ and $\varrho^{r}(v) = m_{\varrho}(v)$ hold for each $v\in V$.
\end{theorem}
\begin{proof}
  We prove that there exists a feasible order if and only if $\sum_{v\in V} m_{\delta}(v)= \sum_{v\in V}m_{\varrho}(v)$ and there exists a solution to the problem with bounds $\delta^{\ell} \leq m_{\delta}$ and $\varrho^{r} \geq m_{\varrho}$.
  This implies the theorem, since the latter problem can be solved in polynomial time by Theorem~\ref{thm:outUpperInLower}.

  To understand the necessity of the condition, observe that both $\sum_{v\in V} m_{\delta}(v)=\sum_{v\in V} \delta^{\ell}(v)$ and $\sum_{v\in V} m_{\varrho}(v)=\sum_{v\in V}\varrho^r(v)$ count the left-going arcs in the desired order. Furthermore, the upper bounds on the left-outdegrees and the lower bounds on the right-indegrees pose weaker constraints than the exact prescriptions.

  To prove the sufficiency, suppose that $\sum_{v\in V} m_{\delta}(v)= \sum_{v\in V}m_{\varrho}(v)$ holds and $\sigma$ is a solution to the problem with bounds $\delta^{\ell} \leq m_{\delta}$ and $\varrho^{r} \geq m_{\varrho}$.
  Then the following holds:
  \[
    \sum_{v\in V} \delta^{\ell}(v)\leq \sum_{v\in V} m_{\delta}(v)= \sum_{v\in V} m_{\varrho}(v)\leq \sum_{v\in V} \varrho^{r}(v).
  \]
  Since both the left- and the right-hand side equal the number of left-going arcs in $\sigma$, the inequalities hold with equality.
  Therefore, $\delta^{\ell}(v)=m_{\delta}(v)$ and $\varrho^{r}(v)=m_{\varrho}(v)$ for each $v\in V$.
  This means that $\sigma$ is a solution to the problem with prescribed left-outdegrees and right-indegrees.
\end{proof}

This result enables the efficient solution of certain ordering problems.
For example, consider the problem of finding an order in which the left-going arcs form an $r$\nobreakdash-in-arborescence and the right-going arcs form an $r$\nobreakdash-out-arborescence.
This corresponds to the degree-bounded ordering problem, where $\varrho^{\ell}(r)=\delta^{\ell}(r)=0$ and $\varrho^{\ell}(v)=\delta^{\ell}(v)=1$ for each vertex $v\in V\setminus \{r\}$.
Since the exact prescriptions for the left-indegrees can equivalently be expressed as exact prescriptions for the right-indegrees, Theorem~\ref{thm:outStrictInStrict} implies polynomial-time solvability.
\begin{corollary}\label{cor:orderingInArbOutArb}
  It can be decided in polynomial time whether the vertices of a digraph can be ordered such that the left-going arcs form an in-arborescence and the right-going arcs form an out-arborescence.
  \FBOX
\end{corollary}

In contrast, the problem of deciding whether a digraph contains an $r$\nobreakdash-in-arborescence and an $r$\nobreakdash-out-arborescence that are arc-disjoint is NP-complete~\cite{bang1991edge}.
However, finding two arc-disjoint $r$\nobreakdash-out-arborescences, or $k$ arc-disjoint $r$\nobreakdash-out-arborescences in general, is polynomial-time solvable~\cite{edmonds1973edge}.
We investigate similar vertex-ordering and arc-partitioning problems in Section~\ref{sec:ArcPartition}.

\section{Ordering and partitioning problems}\label{sec:ArcPartition}
This section explores Problems~\ref{prob:vertexOrdering} and~\ref{prob:arcPartition} for several natural families of acyclic digraphs $\mathcal{F}$, namely, for in-branchings, in-arborescences, matchings, perfect matchings, unions of disjoint dipaths, dipaths, and Hamiltonian dipaths.

Notice that for any acyclic digraph family $\mathcal{F}$, a solution to the vertex-ordering problem naturally provides a solution to the corresponding arc-partitioning problem.
Specifically, one can partition the arcs into left-going and right-going arcs.
In the next two remarks, we demonstrate that in some cases, the reverse direction also holds, but not for all digraph families.

\begin{remark}\label{remark:orderEkviPartition}
If $\mathcal{F}$ is the digraph family of in-branchings, matchings, or unions of disjoint dipaths, then Problems~\ref{prob:vertexOrdering} and~\ref{prob:arcPartition} are essentially the same.
To see this, consider a solution to the arc-partitioning problem in which the in-branching, matching, or disjoint dipaths are inclusion-wise minimal.
In this case, these arcs form an inclusion-wise minimal feedback arc set.
Reversing all of its arcs results in an acyclic graph, whose topological order provides a solution to the vertex-ordering problem.
$\bullet$
\end{remark}

\begin{remark}\label{remark:orderDiffPartition}
  If $\mathcal{F}$ is the digraph family of in-arborescences, perfect matchings, or Hamiltonian dipaths, the vertex-ordering and arc-partitioning problems differ.
  To illustrate this, consider a digraph with two vertices and two parallel arcs.
  In this case, a solution to the arc-partitioning problem can be obtained by partitioning the digraph into two subgraphs, each containing one arc.
  However, for every possible vertex order, either both arcs go to the left or neither of them does, meaning the vertex-ordering problem has no solution.
  $\bullet$
\end{remark}

Similar arc-partitioning problems were discussed in~\cite{bang2022complexity}.
For example, they proved that it is NP-complete, with respect to Turing reduction, to decide whether a digraph can be partitioned into a directed cycle and an acyclic subgraph, or into a directed 2-factor and an acyclic subgraph.
Some of our arc-partitioning problems can be viewed as partitioning into two acyclic subgraphs, with some bounds imposed on the in- or outdegrees.
In~\cite{wood2004bounded}, the authors considered a similar degree-bounded acyclic decomposition problem, where the goal is to partition the digraph into $k$ acyclic subgraphs such that each outdegree is at most $\left\lceil \frac{\delta(v)}{k-1}\right\rceil$.
They proved that every simple digraph can be decomposed this way for any integer $k\ge 2$.

Next, we study the relationship between vertex-ordering problems and their natural arc-partitioning counterparts.
The complexity results are summarized in Table~\ref{tab:partitioningAndOrdering}.

\subsection{In-branchings and in-arborescences}
This section considers Problems~\ref{prob:vertexOrdering} and~\ref{prob:arcPartition} in the case when $\mathcal{F}$ is the family of in-branchings or in-ar\-bores\-cences.
As mentioned in Remark~\ref{remark:orderEkviPartition}, partitioning a digraph into an in-branching and an acyclic subgraph is essentially the same as finding a vertex order in which the left-going arcs form an in-branching.
However, Remark~\ref{remark:orderDiffPartition} shows that when considering the analogous partitioning and vertex-ordering problem for an in-arborescence, the two problems no longer coincide.

Note that partitioning a digraph into an in-branching and an acyclic subgraph can be rephrased as finding an in-branching that covers all directed cycles.
A similar problem was discussed in~\cite[p.~567]{frank2011connections}, where the goal is to cover all directed cuts instead of all directed cycles.
In that work, a necessary and sufficient condition is provided for the existence of an in-branching that covers all directed cuts.
We derive a characterization for the existence of an in-branching that covers all directed cycles, using the $(-\infty, g)$\nobreakdash-bounded ordering problem with $g \equiv 1$.
\begin{theorem}\label{thm:inBranchingPoly}
  It can be decided in polynomial time whether the arc set of a digraph can be partitioned into an in-branching and an acyclic subgraph.
  Such a partition exists if and only if there exists no induced subgraph in which the outdegree of each vertex is at least $2$.
\end{theorem}
\begin{proof}
  The arc-partitioning problem is equivalent to the $(-\infty, g)$\nobreakdash-bounded ordering problem with upper bound $g\equiv 1$.
  Therefore, partitioning into an in-branching and an acyclic subgraph can be solved in polynomial time using Algorithm~\ref{alg:UpperBounds}, and the characterization follows directly from Theorem~\ref{thm:gChar}.
  \hfill\end{proof}

Moreover, the vertex-ordering version of the problem is also polynomial-time solvable, since the two problems coincide.
\begin{corollary}\label{cor:inBranchingOrder}
  It can be decided in polynomial time whether the vertices of a digraph can be ordered such that the left-going arcs form an in-branching.
  \FBOX
\end{corollary}

Furthermore, a similar characterization holds when some vertices are required to be roots in the in-branching.
\begin{theorem}
  For a digraph $D = (V, A)$ and a subset $X \subseteq V$, it can be decided in polynomial time whether the arc set can be partitioned into an acyclic subgraph and an in-branching in which all vertices in $X$ are roots.
  Such a partition exists if and only if there does not exist an induced subgraph $D' = (V', A')$ of $D$ in which the outdegree of each vertex $v \in X$ is at least $1$ and the outdegree of each vertex $v \in V' \setminus X$ is at least $2$.
  \FBOX
\end{theorem}

It is important to note that the in-branching may contain roots other than the vertices in $X$.
Therefore, this theorem is not applicable for partitioning a digraph into an in-arborescence and an acyclic subgraph.
The complexity of this problem remains open.
However, partitioning a digraph into an in-arborescence and a \emph{spanning} acyclic subgraph is NP-complete as a corollary of the hardness proof for Problem~4.2.6 in~\cite{bang2022complexity}.

Note that partitioning into an in-arborescence and an acyclic subgraph is fundamentally different from finding an ordering of the vertices in which the left-going arcs form an in-arborescence.
If the root vertex of the in-arborescence is fixed, then the latter problem is NP-complete by Corollary~\ref{cor:strict}.
Observe that this even implies the hardness of the analogous problem for arbitrary root vertex.
The reduction goes by adding a new vertex with a single incoming arc from the root vertex.
\begin{corollary}\label{cor:inArbOrder}
  It is NP-complete to decide whether the vertices of a digraph can be ordered such that the left-going arcs form an in-arborescence.
  \FBOX
\end{corollary}

Theorem~\ref{thm:inBranchingPoly} characterizes when a digraph can be partitioned into an in-branching and an acyclic subgraph.
However, Theorem~\ref{thm:decMinBoundedOrder} states that it is APX-hard to minimize the size of the in-branching in such a partition (which is essentially the same as minimizing the size of an in-branching whose arcs go to the left in some order of the vertices), through an approximation-preserving reduction from the minimum feedback arc set problem.
It is known that this problem cannot be approximated within a factor of $10\sqrt{5} - 21 \approx 1.3606$, unless P = NP~\cite{dinur2005hardness, karp1972reducibility}, and cannot be approximated within a constant factor if the Unique Games Conjecture holds~\cite{guruswami2008beating}.
The proof of Theorem~\ref{thm:decMinBoundedOrder} implies that the same inapproximability results also apply to minimizing the size of an in-branching that covers all directed cycles.
\begin{corollary}\label{cor:minInBranching}
  In a digraph, it is APX-hard to find a minimum-size in-branching that covers all directed cycles.
  \FBOX
\end{corollary}

Moreover, the APX-hardness of partitioning a digraph into a minimum-cost in-arborescence and an acyclic subgraph follows by an approximation-preserving reduction from partitioning into a minimum-size in-branching and an acyclic subgraph.
The main idea is to add a new vertex $s$ to the digraph with an arc coming in from all other vertices.
Let the cost function be $0$ on these new arcs and $1$ on the rest of the arcs.
This digraph contains an appropriate in-arborescence of cost at most $k$ if and only if the original digraph contains an appropriate in-branching of size at most $k$.
Thus, we obtain the following.
\begin{theorem}
  For a digraph with a $0$\nobreakdash-$1$ cost function on the arc set, it is APX-hard to find a minimum-cost in-arborescence that covers all directed cycles.
  \FBOX
\end{theorem}

Note that approximating the minimum size of a matching that covers all directed cycles is not possible, because we will prove in Theorem~\ref{thm:matchingAndAcyclicNPC} that the decision version of this problem is NP-complete.
The same holds for finding a minimum-cost perfect matching that covers all directed cycles, as shown in Theorem~\ref{thm:perfectMatchingAndAcyclicNPC}.

\subsection{Matchings and perfect matchings}
This section investigates Problems~\ref{prob:vertexOrdering} and~\ref{prob:arcPartition} in the case when $\mathcal{F}$ is the family of matchings or perfect matchings (directed arbitrarily).
Similarly to the case of in-branchings, partitioning a digraph into a matching and an acyclic subgraph is essentially the same as finding an ordering of the vertices in which the left-going arcs form a matching, as shown in Remark~\ref{remark:orderEkviPartition}.
However, if we require the matching to be perfect, then the two problems no longer coincide, as stated in Remark~\ref{remark:orderDiffPartition}.

Motivated by the solvability of partitioning a digraph into an in-branching and an acyclic subgraph, it is natural to ask whether a similar result holds for matchings.
In the following, we show that this problem is NP-complete.

\begin{theorem}~\label{thm:matchingAndAcyclicNPC}
  It is NP-complete to decide whether the arc set of a digraph can be partitioned into a matching and an acyclic subgraph.
\end{theorem}
\begin{proof}
  The problem is clearly in NP.
  The proof is by reduction from the NAE-3-SAT problem, which is known to be NP-complete~\cite{porschen2014xsat}.
  In the NAE-3-SAT problem, the goal is to decide whether a CNF formula, where each clause contains exactly three literals, can be satisfied such that each clause has at least one false literal.

  Let us consider a CNF formula with variables $x_1, \ldots, x_n$ and clauses $c_1 = (c_1^1 \lor c_2^1 \lor c_3^1), \ldots, c_m = (c_1^m \lor c_2^m \lor c_3^m)$.
  We construct a digraph $D$ as follows: For each clause $c_j$, let $D$ contain a gadget $H_{c_j}$ on 9 vertices, denoted by $u_k^j, c_k^j, \overline{c}_k^j$ for $k = 1, 2, 3$.
  The vertex $c_k^j$ corresponds to the $k^\text{th}$ literal of the $j^\text{th}$ clause, and the vertex $\overline{c}_k^j$ corresponds to its negation.
  The gadget contains a directed cycle $c_1^j u_1^j c_2^j u_2^j c_3^j u_3^j$ and a directed cycle $\overline{c}_1^j u_3^j \overline{c}_3^j u_2^j \overline{c}_2^j u_1^j$.
  Figure~\ref{fig:ClauseGadget} illustrates the gadget $H_{c_j}$.
  Moreover, for each variable $x_i$, let $D$ contain a subgraph on the vertices $v_\ell^i$ for $\ell = 1, \ldots, 5$, with an arc $v_1^i v_2^i$, an arc $v_2^i v_1^i$, and a directed cycle $v_2^i v_3^i v_4^i v_5^i$.
  If the clause $c_j$ contains the literals $x_i$ or $\overline{x}_i$, then the subgraphs corresponding to $c_j$ and $x_i$ are connected as follows:
  \begin{enumerate}
  \item If $c_k^j = x_i$, then extend the gadget for $x_i$ with two vertices $y_k^j$ and $z_k^j$, and connect the vertices $c_k^j$ and $\overline{c}_k^j$ from the gadget for $c_j$ with a dipath $v_5^i \overline{c}_k^j z_k^j v_4^i y_k^j c_k^j v_3^i$.
  \item If $c_k^j = \overline{x}_i$, then extend the gadget for $x_i$ similarly, and connect the vertices $c_k^j$ and $\overline{c}_k^j$ from the gadget for $c_j$ with a dipath $v_5^i c_k^j z_k^j v_4^i y_k^j \overline{c}_k^j v_3^i$.
  \end{enumerate}

  Figure~\ref{fig:VariableGadget} illustrates the gadget for variable $x_i$, along with its possible extensions based on whether $c_k^j = x_i$ or $c_k^j = \overline{x}_i$.
  The extended gadget with its extensions is denoted by $H_{x_i}$.

  \begin{figure}[t]
    \centering

    \begin{tikzpicture}[scale=.95]
      \tikzset{VertexStyle/.append style = {minimum size = 20pt,inner sep=0pt}}
      \SetVertexNoLabel
      \grEmptyCycle[form=1,x=0,y=0,RA=2, rotation=30, prefix=k]{6}
      \AssignVertexLabel{k}{$c_1^j$, $u_1^j$, $c_2^j$, $u_2^j$, $c_3^j$, $u_3^j$}
      \grEmptyCycle[form=1,x=0,y=0,RA=1, rotation=30, prefix=b]{3}
      \AssignVertexLabel{b}{$\overline {c}_1^j$, $\overline{c}_2^j$, $\overline{c}_3^j$}

      \draw[\niceArrow] (k0)--(k1);
      \draw[\niceArrow] (k1)--(k2);
      \draw[\niceArrow] (k2)--(k3);
      \draw[\niceArrow] (k3)--(k4);
      \draw[\niceArrow] (k4)--(k5);
      \draw[\niceArrow] (k5)--(k0);

      \draw[\niceArrow] (b0)--(k5);
      \draw[\niceArrow] (k5)--(b2);
      \draw[\niceArrow] (b2)--(k3);
      \draw[\niceArrow] (k3)--(b1);
      \draw[\niceArrow] (b1)--(k1);
      \draw[\niceArrow] (k1)--(b0);

    \end{tikzpicture}
    \caption{The gadget for the clause $c_j$.}\label{fig:ClauseGadget}
    \vspace{5mm}

    \begin{tikzpicture}[xscale=1,yscale=1,scale=1.15]
      \tikzset{VertexStyle/.append style = {minimum size = 20pt,inner sep=0pt}}
      \SetVertexMath
      \Vertex[x=-4, y=2.25, L=v_1^i]{v1}
      \Vertex[x=-4, y=1, L=v_2^i]{v2}
      \Vertex[x=-5, y=0, L=v_3^i]{v3}
      \Vertex[x=-4, y=-1, L=v_4^i]{v4}
      \Vertex[x=-3, y=0, L=v_5^i]{v5}

      \draw[\niceArrow] (v2)--(v3);
      \draw[\niceArrow] (v3)--(v4);
      \draw[\niceArrow] (v4)--(v5);
      \draw[\niceArrow] (v5)--(v2);
      \draw[\niceArrow] (v2) to [out=75,in=-75,looseness=1] (v1);
      \draw[\niceArrow] (v1) to [out=-105,in=105,looseness=1] (v2);

      \Vertex[x=0, y=2.25, L=$ $]{x1}
      \Vertex[x=0, y=1, L=$ $]{x2}
      \Vertex[x=-1, y=0, L=$ $]{x3}
      \Vertex[x=0, y=-1, L=$ $]{x4}
      \Vertex[x=1, y=0, L=$ $]{x5}

      \Vertex[x=-2, y=-1, L=c_k^j]{c0}
      \Vertex[x=-1, y=-2, L=y_k^j]{c1}
      \Vertex[x=1, y=-2, L=z_k^j]{c2}
      \Vertex[x=2, y=-1, L=\overline{c}_k^j]{c3}

      \draw[\niceArrow] (x2)--(x3);
      \draw[\niceArrow] (x3)--(x4);
      \draw[\niceArrow] (x4)--(x5);
      \draw[\niceArrow] (x5)--(x2);
      \draw[\niceArrow] (x4)--(c1);
      \draw[\niceArrow] (c1)--(c0);
      \draw[\niceArrow] (c0)--(x3);
      \draw[\niceArrow] (x5)--(c3);
      \draw[\niceArrow] (c3)--(c2);
      \draw[\niceArrow] (c2)--(x4);

      \draw[\niceArrow] (x2) to [out=75,in=-75,looseness=1] (x1);
      \draw[\niceArrow] (x1) to [out=-105,in=105,looseness=1] (x2);

      \Vertex[x=5, y=2.25, L=$ $]{n1}
      \Vertex[x=5, y=1, L=$ $]{n2}
      \Vertex[x=4, y=0, L=$ $]{n3}
      \Vertex[x=5, y=-1, L=$ $]{n4}
      \Vertex[x=6, y=0, L=$ $]{n5}

      \Vertex[x=3, y=-1, L=\overline{c}_k^j]{nc0}
      \Vertex[x=4, y=-2, L=y_k^j]{nc1}
      \Vertex[x=6, y=-2, L=z_k^j]{nc2}
      \Vertex[x=7, y=-1, L=c_k^j]{nc3}

      \draw[\niceArrow] (n2) to [out=75,in=-75,looseness=1] (n1);
      \draw[\niceArrow] (n1) to [out=-105,in=105,looseness=1] (n2);
      \draw[\niceArrow] (n2)--(n3);
      \draw[\niceArrow] (n3)--(n4);
      \draw[\niceArrow] (n4)--(n5);
      \draw[\niceArrow] (n5)--(n2);
      \draw[\niceArrow] (n4)--(nc1);
      \draw[\niceArrow] (nc1)--(nc0);
      \draw[\niceArrow] (nc0)--(n3);
      \draw[\niceArrow] (n5)--(nc3);
      \draw[\niceArrow] (nc3)--(nc2);
      \draw[\niceArrow] (nc2)--(n4);

    \end{tikzpicture}
    \caption{The gadget for the variable $x_i$ and its extensions when $c_k^j = x_i$ or $c_k^j = \overline{x}_i$, from left to right.
      The gadget has an extension for each clause containing the literal $x_i$ or $\overline{x}_i$.
      Note that the vertices labeled $c_k^j$ and $\overline{c}_k^j$ are identical to those in Figure~\ref{fig:ClauseGadget}.}\label{fig:VariableGadget}
  \end{figure}

  We now prove that the NAE-3-SAT problem is solvable if and only if the digraph $D$ constructed above can be partitioned into a matching and an acyclic subgraph.

  First, suppose the NAE-3-SAT problem is solvable.
  Construct a matching $M$ that covers all directed cycles as follows:
  If $c_k^j$ is true, then add the arc $c_k^j u_k^j$ to $M$; if $c_k^j$ is false, then add the arc $u_k^j \overline{c}_k^j$ to $M$.
  In $H_{x_i}$, cover the 2-cycle between $v_1^i$ and $v_2^i$ with one of its arcs.

  \begin{enumerate}
  \item If $x_i$ is true, then add the arc $v_3^i v_4^i$ to $M$, and if $c_k^j = x_i$, then cover the arc $\overline{c}_k^j z_k^j$; if $c_k^j = \overline{x}_i$, then cover the arc $c_k^j z_k^j$.
  \item If $x_i$ is false, then add the arc $v_4^i v_5^i$ to $M$, and if $c_k^j = x_i$, then cover the arc $y_k^j c_k^j$; if $c_k^j = \overline{x}_i$, then cover the arc $y_k^j \overline{c}_k^j$.
  \end{enumerate}

  Now we show that $M$ covers all directed cycles in $D$.
  Consider the cycles spanned by the gadget $H_{c_j}$:
  The 4-cycles are covered either by the arc going out from $c_k^j$ or by the arc coming in to $\overline{c}_k^j$.
  The two 6-cycles are also covered because clause $c_j$ contains at least one true and at least one false literal.
  If the literal $c_k^j$ is true, then the arc going out from $c_k^j$ covers the outer cycle; if the literal $c_k^j$ is false, then the arc coming in to $\overline{c}_k^j$ covers the inner cycle.

  The cycles inside the gadget $H_{x_i}$ are also covered:
  The 2-cycle between $v_1^i$ and $v_2^i$ is covered by one of its arcs, and the 4-cycle $v_2^i v_3^i v_4^i v_5^i$ is covered either by the arc $v_3^i v_4^i$ (if $x_i$ is true) or by the arc $v_4^i v_5^i$ (if $x_i$ is false).
  For each clause $c_j$ where $c_k^j = x_i$, the 4-cycle $v_3^i v_4^i y_k^j c_k^j$ is covered by either $v_3^i v_4^i$ or $y_k^j c_k^j$, and the 4-cycle $v_4^i v_5^i \overline{c}_k^j z_k^j$ is covered by either $v_4^i v_5^i$ or $\overline{c}_k^j z_k^j$.
  Similarly, for each clause $c_j$ where $c_k^j = \overline{x}_i$, the 4-cycle $v_3^i v_4^i y_k^j \overline{c}_k^j$ is covered by either $v_3^i v_4^i$ or $y_k^j \overline{c}_k^j$, and the 4-cycle $v_4^i v_5^i c_k^j z_k^j$ is covered by either $v_4^i v_5^i$ or $c_k^j z_k^j$.

  The remaining cycles include arcs across different gadgets.
  Such a cycle must contain a dipath in the gadget $H_{x_i}$ for some variable $x_i$, between two vertices corresponding to literals from different clauses, i.e., an $s$\nobreakdash-$t$ dipath in $H_{x_i}$, where $s \in \{c_k^j, \overline{c}_k^j\}$ and $t \in \{c_{k'}^{j'}, \overline{c}_{k'}^{j'}\}$ with $j \neq j'$.

  We now prove that the matching $M$ covers these dipaths.
  First, consider the dipaths going out from the true literal among $c_k^j$ or $\overline{c}_k^j$ (which are literals of the variable $x_i$).
  Regardless of the value of $x_i$, either $c_k^j$ or $\overline{c}_k^j$ is true by definition:
  \begin{enumerate}
  \item If the variable $x_i$ is true, then in $H_{x_i}\setminus M$, the only arc going out from the true among $c_k^j$ or $\overline{c}_k^j$ is going to $v_3^i$, and the only arc going out from $v_3^i$ is the arc $v_3^i v_4^i$, which is covered by $M$.
  \item If the variable $x_i$ is false, then either the literal $c_k^j$ or $\overline{c}_k^j$ is true by definition.
    In $H_{x_i}\setminus M$, the only arc going out from the true among $c_k^j$ or $\overline{c}_k^j$ is going to $z_k^j$, and from $z_k^j$ the only arc is going to $v_4^i$.
    The arc $v_4^i v_5^i$ is covered by the matching, so from $v_4^i$ the only free arcs are going to $y_{k'}^{j'}$ for each clause $c_{j'}$ in which the literal $c_{k'}^{j'}$ is $x_i$ or $\overline{x}_i$.
    Finally, the arc going out from such a vertex $y_{k'}^{j'}$ is covered by $M$.
  \end{enumerate}

  Similarly, consider the dipaths going out from the false literal among $c_k^j$ or $\overline{c}_k^j$, which are both literals of the variable $x_i$.
  \begin{enumerate}
  \item If the variable $x_i$ is true, then in $H_{x_i} \setminus M$, the only arc going out from the false among $c_k^j$ or $\overline{c}_k^j$ is going to $z_k^j$ and this arc is covered by $M$.
  \item If the variable $x_i$ is false, then in $H_{x_i} \setminus M$, the only arc going out from the false among $c_k^j$ or $\overline{c}_k^j$ is going to $v_3^i$, and from $v_3$ the only arc goes to $v_4^i$.
    The arc $v_4^i v_5^i$ is covered by $M$, and from $v_4^i$, the only free arcs go to $y_{k'}^{j'}$ for each clause $c_{j'}$ in which the literal $c_{k'}^{j'}$ is $x_i$ or $\overline{x}_i$.
    Finally, the arc going out from such a vertex $y_{k'}^{j'}$ is covered by $M$.
  \end{enumerate}

  Thus, $M$ covers all directed cycles in $D$, so $D \setminus M$ is acyclic.

  Conversely, suppose there is a matching $M$ that covers all directed cycles in $D$.
  Notice that in $H_{x_i}$, the matching contains one arc between $v_1^i$ and $v_2^i$, and covers the cycle $v_2^i v_3^i v_4^i v_5^i$.
  Thus, either $v_3^i v_4^i \in M$ or $v_4^i v_5^i \in M$.
  If $v_3^i v_4^i \in M$, then set the variable $x_i$ to true; if $v_4^i v_5^i \in M$, set $x_i$ to false.
  We now show that this leads to a solution of the NAE-3-SAT problem.

  If $v_3^i v_4^i \in M$, then for each clause $c_j$ where $c_k^j = x_i$, the matching must include an arc from the cycle $v_4^i v_5^i \overline{c}_k^j z_k^j$, and this arc must be incident to $\overline{c}_k^j$.
  For each clause $c_j$ where $c_k^j = \overline{x}_i$, the matching must include an arc from the cycle $v_4^i v_5^i c_k^j z_k^j$, and this arc must be incident to $c_k^j$.
  Similarly, if  $v_4^iv_5^i\in M$, then for each clause $c_j$ where $c_k^j=x_i$, the matching must include an arc from the cycle $v_3^iv_4^iy_k^jc_k^j$, and this arc must be incident to $c_k^j$; and for each clause $c_j$ where $c_k^j=\overline{x}_i$ the matching must include an arc from the cycle $v_3^iv_4^iy_k^j\overline{c}_k^j$, and this arc must be incident to $\overline{c}_k^j$.
  In other words, for each clause $c_j$ for which the literal $c_k^j$ is $x_i$ or $\overline{x}_i$, there is an arc in $M\cap H_{x_i}$ incident to the false among $c_k^j$ and $\overline{c}_k^j$.

  Finally, consider the gadget $H_{c_j}$ for each clause $c_j$.
  Notice that $M$ contains at least one arc incident to $c_k^j$ or $\overline{c}_k^j$, because these arcs form a 4-cycle.
  This arc is incident to the true literal among $c_k^j$ or $\overline{c}_k^j$, since $M$ already contains an arc in $H_{x_i}$ incident to the other literal.
  Furthermore, $M$ contains at least one arc from the outer cycle $c_1^j u_1^j c_2^j u_2^j c_3^j u_3^j$ and at least one arc from the inner cycle $\overline{c}_1^j u_3^j \overline{c}_3^j u_2^j \overline{c}_2^j u_1^j$.
  This guarantees that, in each clause $c_j$, at least one literal is true and at least one literal is false.

  Therefore, a solution to the NAE-3-SAT problem can be found by setting $x_i$ to true when the arc $v_3^i v_4^i$ is part of the matching, and false otherwise.
\end{proof}

This, together with Remark~\ref{remark:orderEkviPartition}, implies that the vertex-ordering version of the problem is also NP-complete.
\begin{corollary}\label{cor:matchingAcyclic}
  It is NP-complete to decide whether the vertices of a digraph can be ordered such that the left-going arcs form a matching.
  \FBOX
\end{corollary}

By a simple modification of this construction, it follows that the problem of partitioning a digraph into a matching and an acyclic subgraph is also NP-complete when the matching is required to be perfect.
\begin{theorem}~\label{thm:perfectMatchingAndAcyclicNPC}
  It is NP-complete to decide whether the arc set of a digraph can be partitioned into a perfect matching and an acyclic subgraph.
\end{theorem}
\begin{proof}
  The problem is clearly in NP.
  The proof is by reduction from the NAE-3-SAT problem, which is known to be NP-complete~\cite{porschen2014xsat}.
  We use the same digraph $D$ as in the proof of Theorem~\ref{thm:matchingAndAcyclicNPC}, with a small modification.
  The gadget $H_{c_j}$ corresponding to the clause $c_j$ remains the same as shown in Figure~\ref{fig:ClauseGadget}.
  For each variable $x_i$, we add a new vertex $s^i$ and a dipath $v_3^i s^i v_5^i$ to the gadget corresponding to $x_i$.
  Moreover, for each clause $c_j$ where the literal $c_k^j$ is $x_i$ or $\overline{x}_i$, we add a new vertex $t_k^j$ and a dipath $y_k^j t_k^j z_k^j$ to the gadget corresponding to $x_i$.
  Figure~\ref{fig:ModifiedVariableGadget} illustrates the modified gadget corresponding to the variable $x_i$, denoted by $H'_{x_i}$.
  Let the modified digraph, containing the gadget $H_{c_j}$ for each clause $c_j$ and the gadget $H'_{x_i}$ for each variable $x_i$, be denoted by $D'$.

  \begin{figure}[t]
    \centering
    \begin{tikzpicture}[xscale=1,yscale=1,scale=1.15]
      \tikzset{VertexStyle/.append style = {minimum size = 20pt,inner sep=0pt}}
      \SetVertexMath
      \Vertex[x=-4, y=2.25, L=v_1^i]{v1}
      \Vertex[x=-4, y=1, L=v_2^i]{v2}
      \Vertex[x=-5, y=0, L=v_3^i]{v3}
      \Vertex[x=-4, y=-1, L=v_4^i]{v4}
      \Vertex[x=-3, y=0, L=v_5^i]{v5}

      \draw[\niceArrow] (v2)--(v3);
      \draw[\niceArrow] (v3)--(v4);
      \draw[\niceArrow] (v4)--(v5);
      \draw[\niceArrow] (v5)--(v2);
      \draw[\niceArrow] (v2) to [out=75,in=-75,looseness=1] (v1);
      \draw[\niceArrow] (v1) to [out=-105,in=105,looseness=1] (v2);

      \Vertex[x=0, y=2.25, L=$ $]{x1}
      \Vertex[x=0, y=1, L=$ $]{x2}
      \Vertex[x=-1, y=0, L=$ $]{x3}
      \Vertex[x=0, y=-1, L=$ $]{x4}
      \Vertex[x=1, y=0, L=$ $]{x5}

      \Vertex[x=-2, y=-1, L=c_k^j]{c0}
      \Vertex[x=-1, y=-2, L=y_k^j]{c1}
      \Vertex[x=1, y=-2, L=z_k^j]{c2}
      \Vertex[x=2, y=-1, L=\overline{c}_k^j]{c3}

      \draw[\niceArrow] (x2)--(x3);
      \draw[\niceArrow] (x3)--(x4);
      \draw[\niceArrow] (x4)--(x5);
      \draw[\niceArrow] (x5)--(x2);
      \draw[\niceArrow] (x4)--(c1);
      \draw[\niceArrow] (c1)--(c0);
      \draw[\niceArrow] (c0)--(x3);
      \draw[\niceArrow] (x5)--(c3);
      \draw[\niceArrow] (c3)--(c2);
      \draw[\niceArrow] (c2)--(x4);

      \draw[\niceArrow] (x2) to [out=75,in=-75,looseness=1] (x1);
      \draw[\niceArrow] (x1) to [out=-105,in=105,looseness=1] (x2);

      \Vertex[x=5, y=2.25, L=$ $]{n1}
      \Vertex[x=5, y=1, L=$ $]{n2}
      \Vertex[x=4, y=0, L=$ $]{n3}
      \Vertex[x=5, y=-1, L=$ $]{n4}
      \Vertex[x=6, y=0, L=$ $]{n5}

      \Vertex[x=3, y=-1, L=\overline{c}_k^j]{nc0}
      \Vertex[x=4, y=-2, L=y_k^j]{nc1}
      \Vertex[x=6, y=-2, L=z_k^j]{nc2}
      \Vertex[x=7, y=-1, L=c_k^j]{nc3}

      \draw[\niceArrow] (n2) to [out=75,in=-75,looseness=1] (n1);
      \draw[\niceArrow] (n1) to [out=-105,in=105,looseness=1] (n2);
      \draw[\niceArrow] (n2)--(n3);
      \draw[\niceArrow] (n3)--(n4);
      \draw[\niceArrow] (n4)--(n5);
      \draw[\niceArrow] (n5)--(n2);
      \draw[\niceArrow] (n4)--(nc1);
      \draw[\niceArrow] (nc1)--(nc0);
      \draw[\niceArrow] (nc0)--(n3);
      \draw[\niceArrow] (n5)--(nc3);
      \draw[\niceArrow] (nc3)--(nc2);
      \draw[\niceArrow] (nc2)--(n4);

      \Vertex[x=-4, y=0, L=s^i]{s}
      \Vertex[x=0, y=0, L=$ $]{xs}
      \Vertex[x=0, y=-3, L=t_k^j]{xt}
      \Vertex[x=5, y=0, L=$ $]{ns}
      \Vertex[x=5, y=-3, L=t_k^j]{nt}

      \draw[\niceArrow] (v3)--(s);
      \draw[\niceArrow] (s)--(v5);
      \draw[\niceArrow] (x3)--(xs);
      \draw[\niceArrow] (xs)--(x5);
      \draw[\niceArrow] (c1)--(xt);
      \draw[\niceArrow] (xt)--(c2);
      \draw[\niceArrow] (n3)--(ns);
      \draw[\niceArrow] (ns)--(n5);
      \draw[\niceArrow] (nc1)--(nt);
      \draw[\niceArrow] (nt)--(nc2);

    \end{tikzpicture}
    \caption{The modified gadget belonging to the variable $x_i$ and its extensions if the literal $c_k^j$ is $x_i$ or $\overline{x}_i$, from left to right.
      The gadget has an extension for each clause containing the literals $x_i$ or $\overline{x}_i$.
      The vertices with labels $c_k^j$ and $\overline{c}_k^j$ are identical to those with the same label in Figure~\ref{fig:ClauseGadget}.}\label{fig:ModifiedVariableGadget}
  \end{figure}
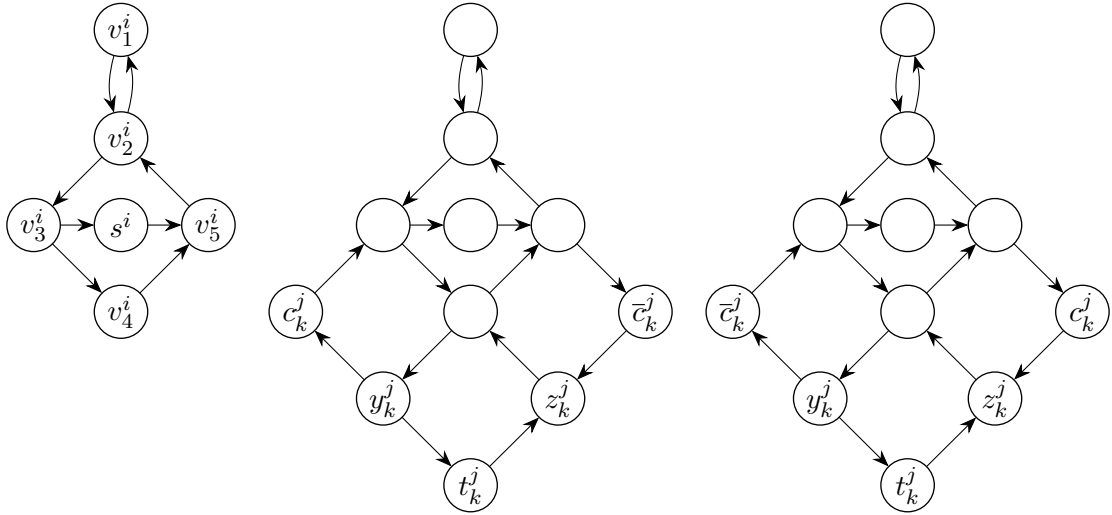
  Now, we show that $D'$ can be partitioned into a perfect matching and an acyclic subgraph if and only if the instance of the NAE-3-SAT problem is solvable.

  First, suppose there exists a perfect matching covering all directed cycles in $D'$.
  Then, this matching, when restricted to $D$, covers all directed cycles in $D$.
  Thus, we can construct a solution to the instance of the NAE-3-SAT problem as in the proof of Theorem~\ref{thm:matchingAndAcyclicNPC}.

  Second, suppose there exists a solution to the NAE-3-SAT problem.
  We construct the perfect matching $M$ covering all directed cycles in $D'$ as follows.
  Observe that the perfect matching uses the same arcs from $D$ as in the proof of Theorem~\ref{thm:matchingAndAcyclicNPC}, and we only add some new arcs incident to the newly added vertices $s^i$ and $t_k^j$.
  If the literal $c_k^j$ is true, then we add the arc $c_k^j u_k^j$ to $M$; if the literal $c_k^j$ is false, then we add the arc $u_k^j \overline{c}_k^j$ to $M$.
  In the gadget $H_{x_i}$, we cover the 2-cycle between $v_1^i$ and $v_2^i$ with one of its arcs.

  \begin{enumerate}
  \item If the variable $x_i$ is true, then we add the arcs $v_3^i v_4^i$, $s^i v_5^i$, and $y_k^j t_k^j$ to $M$.
    Furthermore, if $c_k^j = x_i$, then we cover the arc $\overline{c}_k^j z_k^j$; if $c_k^j = \overline{x}_i$, then we cover the arc $c_k^j z_k^j$.
  \item If the variable $x_i$ is false, then we add the arcs $v_4^i v_5^i$, $v_3^i s^i$, and $t_k^j z_k^j$ to $M$.
    Furthermore, if $c_k^j = x_i$, then we cover the arc $y_k^j c_k^j$; if $c_k^j = \overline{x}_i$, then we cover the arc $y_k^j \overline{c}_k^j$.
  \end{enumerate}

  It is easy to see that $M$ is a perfect matching.
  Now, we argue that $M$ covers all directed cycles in $D'$.
  Clearly, $M$ covers all directed cycles in $D$ because $M$, when restricted to $D$, is the same as the matching constructed in the proof of Theorem~\ref{thm:matchingAndAcyclicNPC}, which does cover all directed cycles in $D$.
  Every directed cycle in $D'$ that is not entirely contained in $D$ must involve at least one of the newly added dipaths $v_3^i s^i v_5^i$ or $y_k^j t_k^j z_k^j$.
  Note that $M$ covers one of the arcs in each of these dipaths, so it covers all directed cycles.
  This completes the proof.
\end{proof}

Recall that partitioning a digraph into a perfect matching and an acyclic subgraph is significantly different from deciding whether a vertex order exists such that the left-going arcs form a perfect matching, see Remark~\ref{remark:orderDiffPartition}.
Now, we prove that the latter problem is also NP-complete.
The proof relies on the construction given in the proof of Theorem~\ref{thm:perfectMatchingAndAcyclicNPC}.
\begin{theorem}~\label{thm:perfectMatchinOrderNPC}
  It is NP-complete to decide whether the vertices of a digraph can be ordered such that the left-going arcs form a perfect matching.
\end{theorem}
\begin{proof}
  The problem is clearly in NP.
  The proof is by reduction from the NAE-3-SAT problem, which is known to be NP-complete~\cite{porschen2014xsat}.
  We use the same digraph $D'$ as in the proof of Theorem~\ref{thm:perfectMatchingAndAcyclicNPC}.
  Figures~\ref{fig:ClauseGadget} and \ref{fig:ModifiedVariableGadget} illustrate the gadgets $H_{c_j}$ corresponding to the clause $c_j$ and $H'_{x_i}$ corresponding to the variable $x_i$, respectively.

  We now show that there exists an order of the vertices of $D'$ in which the left-going arcs form a perfect matching if and only if the instance of the NAE-3-SAT problem is solvable.

  First, suppose there exists such a vertex order.
  Clearly, we can partition $D'$ into a perfect matching and an acyclic subgraph by partitioning into the left-going arcs and the right-going arcs.
  Therefore, from the proof of Theorem~\ref{thm:perfectMatchingAndAcyclicNPC}, it follows that the instance of the NAE-3-SAT problem is solvable.

  Second, suppose the instance of the NAE-3-SAT problem is solvable.
  Consider the perfect matching $M$ constructed in the proof of Theorem~\ref{thm:perfectMatchingAndAcyclicNPC}.
  We show that there exists a vertex order in which exactly the arcs in $M$ are left-going, by proving that $M$ is an inclusion-wise minimal feedback arc set.
  For each arc $a$ in $M$, we provide a cycle in $D'$ that is covered only by $a$.

  In the gadget $H_{c_j}$, $M$ contains exactly one arc from each 4-cycle, which is either the arc going out from $c_k^j$ or the arc coming in to $\overline{c}_k^j$.
  In the gadget $H_{x_i}$, $M$ contains exactly one of the arcs in the 2-cycle between $v_1^i$ and $v_2^i$.
  \begin{enumerate}
  \item If the variable $x_i$ is true, then $v_3^i v_4^i$ is the only arc in $M$ from the cycle $v_2^i v_3^i v_4^i v_5^i$, $s^iv_5^i$ is the only arc in $M$ from the cycle $v_2^i v_3^i s^i v_5^i$, and $y_k^jt_k^j$ is the only arc in $M$ from the cycle $v_4^i y_k^j t_k^j z_k^j$.
    Furthermore, if the literal $c_k^j = x_i$, then $\overline{c}_k^j z_k^j$ is the only arc in $M$ from the cycle $v_4^i v_5^i \overline{c}_k^j z_k^j$; if the literal $c_k^j = \overline{x}_i$, then $c_k^j z_k^j$ is the only arc in $M$ from the cycle $v_4^i v_5^i c_k^j z_k^j$.
  \item Otherwise, if the variable $x_i$ is false, then $v_4^i v_5^i$ is the only arc in $M$ from the cycle $v_2^i v_3^i v_4^i v_5^i$, $v_3^is^i$ is the only arc in $M$ from the cycle $v_2^i v_3^i s^i v_5^i$, and $t_k^jz_k^j$ is the only arc in $M$ from the cycle on $v_4^i y_k^j t_k^j z_k^j$.
    Furthermore, if the literal $c_k^j = x_i$, then $y_k^j c_k^j$ is the only arc in $M$ from the cycle $v_3^i v_4^i y_k^j c_k^j$; if the literal $c_k^j = \overline{x}_i$, then $y_k^j \overline{c}_k^j$ is the only arc in $M$ from the cycle $v_3^i v_4^i y_k^j \overline{c}_k^j$.
  \end{enumerate}
  This proves that $M$ is an inclusion-wise minimal feedback arc set.
  Therefore, in the topological order of $D' \setminus M$, exactly the arcs of $M$ are going to the left.
\end{proof}

\subsection{Hamiltonian and disjoint dipaths}
This section explores Problems~\ref{prob:vertexOrdering} and~\ref{prob:arcPartition} in the case when $\mathcal{F}$ is the family of unions of disjoint dipaths, dipaths, or Hamiltonian dipaths.
In the case of disjoint dipaths, the arc-partitioning and vertex-ordering problems are equivalent, similar to the situations involving in-branchings and matchings, see Remark~\ref{remark:orderEkviPartition}.
However, for the Hamiltonian dipath case, the two problems no longer coincide, much like the case of in-arborescences and perfect matchings, see Remark~\ref{remark:orderDiffPartition}.

We begin with the problem of partitioning a digraph into disjoint dipaths and an acyclic subgraph.
\begin{theorem}\label{thm:disjPathAcyclic}
  It is NP-complete to decide whether the arc set of a digraph can be partitioned into disjoint dipaths and an acyclic subgraph.
\end{theorem}
\begin{proof}
  We prove this result by reduction from the problem of partitioning a digraph into a matching and an acyclic subgraph, which is NP-complete by Theorem~\ref{thm:matchingAndAcyclicNPC}.
  Let $D = (V, A)$ be a digraph for which we want to decide whether it can be partitioned into a matching and an acyclic subgraph.
  We construct a new digraph $D'$ as follows:
  Let $D'$ initially be the same as $D$.
  For each vertex $v \in V$, introduce a copy of $v$, denoted by $v'$, and add a directed cycle of length two between $v$ and $v'$ (i.e., add the directed arcs $v v'$ and $v' v$).
  We now show that $D$ can be partitioned into a matching and an acyclic subgraph if and only if $D'$ can be partitioned into disjoint dipaths and an acyclic subgraph.

  First, suppose $M$ is a matching in $D$ such that $D \setminus M$ is acyclic.
  We construct a union of disjoint dipaths $P$ in $D'$ as follows.
  We first include the arcs of $M$ in $P$.
  For each vertex $v \in V$, if $v$ has an incoming arc in $M$, add the directed arc $v v'$ to the dipaths; otherwise, add the directed arc $v' v$.
  It is easy to see that the set $P$ of resulting arcs forms a union of disjoint dipaths, and these dipaths cover all directed cycles in $D'$.
  This is because $M$ covers all directed cycles in the original digraph $D$, and the length-two cycles between each $v$ and $v'$ in $D'$ are covered by the arcs in $P$, ensuring that no directed cycle remains uncovered.

  Second, suppose $P$ is a union of disjoint dipaths in $D'$ such that $D' \setminus P$ is acyclic.
  We now show that the restriction of $P$ to the original digraph $D$ forms a matching $M$ that covers all directed cycles in $D$.
  It is easy to verify that $M$ covers all directed cycles in $D$ because it is the restriction of $P$ to $D$ and $P$ covers all directed cycles in $D'$.
  Moreover, for each vertex $v \in V$, the union of disjoint dipaths $P$ must include either the arc $v v'$ or the arc $v' v$, covering the length-two cycle between $v$ and $v'$.
  Since $P$ consists of disjoint dipaths, it can only contain one other arc incident to $v$ from the original arc set, ensuring that the restriction of $P$ forms a matching in $D$.
  Thus, the restriction of $P$ to $D$ is a matching that covers all directed cycles in $D$, completing the proof.
\end{proof}

This implies that the vertex-ordering version of the problem is also NP-complete since the ordering problem for disjoint dipaths is essentially equivalent to the corresponding partitioning problem by Remark~\ref{remark:orderEkviPartition}.
\begin{corollary}\label{cor:disjPathOrder}
  It is NP-complete to decide whether the vertices of a digraph can be ordered such that the left-going arcs form disjoint dipaths.
  \FBOX
\end{corollary}

Next, we examine the complexity of the analogous arc-partitioning and vertex-ordering problems in the case of a Hamiltonian dipath.
In this case, the two problems are no longer equivalent, as demonstrated in Remark~\ref{remark:orderDiffPartition}.
\begin{theorem}\label{thm:hamPathAcyclic}
  It is NP-complete to decide whether the arc set of a digraph can be partitioned into a Hamiltonian dipath and an acyclic subgraph.
\end{theorem}
\begin{proof}
  We prove by reduction from the Hamiltonian dipath problem, which is known to be NP-complete~\cite{karp1972reducibility}.
  Let $D' = (V', A')$ be a digraph for which we want to decide whether it contains a Hamiltonian dipath.
  We construct a new digraph $D = (V_1, V_2; A)$ by splitting the vertices of $D'$.
  Take two copies of the vertex set $V'$, denoted $V_1$ and $V_2$.
  For each vertex $v' \in V'$, let $v_1$ and $v_2$ denote the corresponding copies of $v'$ in $V_1$ and $V_2$, respectively.
  Add a directed cycle of length two between $v_1$ and $v_2$ for each vertex $v' \in V'$.
  For each arc $u'v' \in A'$, add a directed arc from $u_2$ to $v_1$ in $D$.
  We claim that $D'$ contains a Hamiltonian dipath if and only if $D$ can be partitioned into a Hamiltonian dipath and an acyclic subgraph.

  First, suppose $P'$ is a Hamiltonian dipath in $D'$.
  We construct a Hamiltonian dipath $P$ in $D$:
  For each vertex $v' \in V'$, include the corresponding arc $v_1 v_2 \in A$ in $P$.
  For each arc $u' v' \in P'$, include the corresponding arc $u_2 v_1 \in A$ in $P$.
  We now show that the constructed Hamiltonian dipath $P$ covers all directed cycles in $D$.
  The subgraphs induced by $V_1$ and $V_2$ in $D$ contain no arcs, so any directed cycle in $D$ must contain at least one arc from $V_1$ to $V_2$.
  The only arcs from $V_1$ to $V_2$ are the arcs $v_1 v_2$ between the two copies of each vertex, which are included in $P$.
  Therefore, $P$ covers all directed cycles in $D$.

  Second, suppose $D$ can be partitioned into a Hamiltonian dipath $P$ and an acyclic subgraph.
  Since $P$ is a Hamiltonian dipath in $D$, it must include one arc between $v_1$ and $v_2$ for each $v' \in V'$.
  We can assume that it is the arc $v_1 v_2$, because for each vertex $v_1 \in V_1$, the only outgoing arc is $v_1 v_2 \in A$, and for each vertex $v_2 \in V_2$, the only incoming arc is $v_1 v_2 \in A$.
  Now consider the arcs in $P$ that go from $V_2$ to $V_1$.
  The corresponding arcs $\{ u' v' \in A' : u_2 v_1 \in P \}$ form a Hamiltonian dipath in $D'$.
  Thus, there exists a Hamiltonian dipath in $D'$ if and only if $D$ can be partitioned into a Hamiltonian dipath and an acyclic subgraph, completing the proof.
\end{proof}

The NP-completeness of partitioning a digraph into a dipath and an acyclic subgraph can be established using a similar proof as in Theorem~\ref{thm:hamPathAcyclic}.
\begin{theorem}\label{thm:pathAcyclic}
  It is NP-complete to decide whether the arc set of a digraph can be partitioned into a dipath and an acyclic subgraph.
  \FBOX
\end{theorem}

In contrast, the vertex-ordering version for a Hamiltonian dipath can be solved in polynomial time.
\begin{theorem}\label{thm:hamPathOrder}
  It can be decided in polynomial time whether the vertices of a digraph can be ordered such that the left-going arcs form a Hamiltonian dipath.
\end{theorem}
\begin{proof}
  For a digraph $D=(V,A)$, the vertex orders in which the left-going arcs form a Hamiltonian $s$\nobreakdash-$t$ dipath can be characterized as follows:
  The last vertex in the order is $s$, with $\delta^\ell(s) = 1$ and $\varrho^r(s) = 0$.
  The first vertex is $t$, with $\delta^\ell(t) = 0$ and $\varrho^r(t) = 1$.
  Furthermore, $\delta^\ell(v) = \varrho^r(v) = 1$ holds for each vertex $v \in V \setminus \{s,t\}$.

  This characterization transforms the problem of determining whether there exists a vertex order such that the left-going arcs form a Hamiltonian $s$\nobreakdash-$t$ dipath into the problem of finding a vertex order with simultaneous exact bounds for both the left-outdegree and right-indegree of each vertex.
  By Theorem~\ref{thm:outStrictInStrict}, this problem is solvable in polynomial time.
  Therefore, we can decide whether the sought vertex order exists by invoking the polynomial-time algorithm for solving the problem of finding a vertex order with simultaneous exact bounds for left-outdegree and right-indegree for every distinct $s,t\in V$.
\end{proof}

Observe that the same approach does not work for finding an order in which the left-going arcs form an $s$\nobreakdash-$t$ dipath, as opposed to a Hamiltonian $s$\nobreakdash-$t$ dipath.
In the Hamiltonian case, the degree constraints are straightforward provided that the first and the last vertices are given, but for an $s$\nobreakdash-$t$ dipath, the situation is more complicated.
In this case, we have two possibilities for each vertex $v \in V \setminus \{s,t\}$:
If $v$ is part of the $s$\nobreakdash-$t$ dipath, then $\delta^\ell(v) = \varrho^r(v) = 1$.
Otherwise, $v$ is not on the dipath, and we have $\delta^\ell(v) = \varrho^r(v) = 0$.
Thus, the problem is no longer directly solvable using Theorem~\ref{thm:outStrictInStrict}.
However, we show that this problem is still solvable in polynomial time by presenting an algorithm for a more general problem.

Let $D = (V, A)$ be a digraph and let $S, T \subseteq V$ be two disjoint subsets with $|S| = |T| = k$.
We aim to decide whether there exists an order in which the left-going arcs form $k$ disjoint $S$\nobreakdash-$T$ dipaths.

\begin{algorithm}[h]
  \caption{\hspace{0.5cm}\textsc{Order with $k$\nobreakdash-disjoint $S$\nobreakdash-$T$ dipaths going to the left}}\label{alg:kDisjointPaths}
  \begin{algorithmic}[1]
    \State $V' \coloneqq V$, $n \coloneqq |V|$, $X \coloneqq S$, $Y \coloneqq \emptyset$
    \State Let $\sigma_1, \ldots, \sigma_n$ denote the vertex order we are searching for.
    \For{$i = n, \ldots, 1$}
    \If{$\exists \, v \in X$ with $\delta(v, V' \setminus \{v\}) \leq 1$}
    \State $\sigma_i \coloneqq v$, $V' \coloneqq V' \setminus \{v\}$\label{alg:line:fixX}
    \If{$\delta(v, V' \setminus \{v\}) = 1$}
    \State Let $u$ be the only out-neighbor of $v$ in $V'$.
    \If{$u \in T$}
    \State $X \coloneqq X \setminus \{v\}$, $Y \coloneqq Y \cup \{u\}$
    \Else
    \State $X \coloneqq (X \setminus \{v\}) \cup \{u\}$
    \EndIf
    \Else
    \State $X \coloneqq X \setminus \{v\}$
    \EndIf
    \If{$|X \cup Y| \leq k - 1$}
    \State \textbf{return} \textit{No solution exists}\label{alg:line:cutSet}
    \EndIf
    \ElsIf{$\exists \, v \in V' \setminus ((X \cup T) \setminus Y)$ with $\delta(v, V' \setminus \{v\}) = 0$}
    \State $\sigma_i \coloneqq v$, $V' \coloneqq V' \setminus \{v\}$\label{alg:line:fix0}
    \Else
    \State \textbf{return} \textit{No solution exists}\label{alg:line:certificate}
    \EndIf
    \EndFor
    \State \textbf{return} $\sigma_1, \ldots, \sigma_n$
  \end{algorithmic}
\end{algorithm}

Algorithm~\ref{alg:kDisjointPaths} works by iteratively fixing vertices from right to left.
$V'$ denotes the set of the non-fixed vertices, $X$ denotes the set of those vertices from $V'\setminus T$ that either are in $S$ or have a fixed in-neighbor (i.e., vertices that must be fixed with left-outdegree $\delta^\ell(v) = 1$) and $Y$ denotes the vertices from $T$ that already have a fixed in-neighbor.
At each iteration, we either fix a vertex from $X$ with left-outdegree $\delta^{\ell}(v) \leq 1$ (Line~\ref{alg:line:fixX}) or a vertex from $V'\setminus((X\cup T)\setminus Y)$ with left-outdegree $\delta^{\ell}(v)=0$ (Line~\ref{alg:line:fix0}).
When a vertex is fixed, we update the sets $V'$, $X$, and $Y$.
Throughout the algorithm, we maintain that $X\cup Y$ covers all $S$\nobreakdash-$T$ dipaths.
If the set $X \cup Y$ ever becomes too small, indicating that fewer than $k$ disjoint $S$\nobreakdash-$T$ dipaths can be formed, then the algorithm returns that no solution exists (Line~\ref{alg:line:cutSet}).
If at any point no vertex can be fixed, the algorithm concludes that no solution exists (Line~\ref{alg:line:certificate}).

\begin{theorem}\label{thm:kPathOrder}
  For a digraph $D = (V, A)$ and two disjoint subsets $S, T \subseteq V$ of size $k$, Algorithm~\ref{alg:kDisjointPaths} decides in polynomial time whether there exists an ordering of the vertices such that the left-going arcs form $k$ disjoint $S$\nobreakdash-$T$ dipaths.
\end{theorem}
\begin{proof}
  The running time of Algorithm~\ref{alg:kDisjointPaths} is clearly polynomial, so we focus on proving its correctness.

  First, we show that if Algorithm~\ref{alg:kDisjointPaths} produces a feasible order, then the left-going arcs form $k$ disjoint $S$\nobreakdash-$T$ dipaths.
  During each iteration of the algorithm, the following sets are maintained: $V'$ is the set of non-fixed vertices, $X$ is the set of non-fixed vertices from $S$ and those vertices in $V \setminus T$ that have a fixed in-neighbor, and $Y$ is the set of vertices in $T$ that already have a fixed in-neighbor.

  Now, we establish bounds on the left-outdegree and right-indegree of each vertex in the order produced by the algorithm.
  For any vertex $s \in S$, we have
  \[
    \varrho^r(s) \geq 0 \quad \text{and} \quad \delta^\ell(s) \leq 1,
  \]
  since each vertex in $S$ can only be fixed in Line~\ref{alg:line:fixX} with left-outdegree at most one.

  Next, for any vertex $t \in T$, we observe that
  \[
    \varrho^r(t) \geq 1 \quad \text{and} \quad \delta^\ell(t) = 0,
  \]
  hold because each vertex $t\in T$ can only be fixed in Line~\ref{alg:line:fix0} if $t \in Y$ with right-indegree at least one and left-outdegree zero.

  For any vertex $v \in V \setminus (S \cup T)$, the algorithm can only fix $v$ if $\delta^\ell(v) \leq 1$.
  Furthermore, if $v$ is fixed with $\delta^\ell(v) = 1$ (in Line~\ref{alg:line:fixX}), then $v \in X$, and therefore, $\varrho^r(v) \geq 1$.
  Hence, for each vertex $v \in V \setminus (S \cup T)$, we have
  \[
    \varrho^r(v) \geq \delta^\ell(v).
  \]

  Combining the observations above, we obtain that
  \begin{align*}
    \sum_{v \in V} \varrho^{r}(v) &= \sum_{s\in S}\varrho^{r}(s)+\sum_{t\in T}\varrho^{r}(t)+\sum_{v\in V\setminus(S\cup T)}\varrho^{r}(v) \geq 0+k+\sum_{v\in V\setminus(S\cup T)}\delta^{\ell}(v)\\
                                  &\geq  \sum_{t\in T}\delta^{\ell}(t)+\sum_{s\in S}\delta^{\ell}(s)+\sum_{v\in V\setminus(S\cup T)}\delta^{\ell}(v)=\sum_{v\in V} \delta^{\ell}(v).
  \end{align*}
  In fact, all inequalities hold with equality, because both $\sum_{v\in V} \varrho^r(v)$ and $\sum_{v\in V} \delta^\ell(v)$ are equal to the number of left-going arcs, therefore, the left-hand side and the right-hand side of the chain of inequalities are equal.

  This implies that the bounds above also hold with equality, that is, $\varrho^{r}(s)=0$ and $\delta^{\ell}(s)=1$ for each vertex $s\in S$, $\varrho^{r}(t)=1$ and $\delta^{\ell}(t)=0$ for each vertex $t\in T$, and $\varrho^{r}(v)=\delta^{\ell}(v)\leq 1$ for each vertex $v\in V\setminus(S\cup T)$.
  Therefore, in the order produced by Algorithm~\ref{alg:kDisjointPaths}, the left-going arcs indeed form $k$ disjoint $S$\nobreakdash-$T$ dipaths.

  Second, we show that if the algorithm terminates before finding a complete order, then no solution exists.
  We need the following claim.
  \begin{claim}\label{cl:cutSet}
    At the beginning of each iteration of Algorithm~\ref{alg:kDisjointPaths}, the set $X \cup Y$ covers all $S$\nobreakdash-$T$ dipaths.
  \end{claim}
  \begin{subproof}
    Consider an $S$\nobreakdash-$T$ dipath $P$ from $s \in S$ to $t \in T$.
    If all vertices of $P$ are fixed, then $t \in Y$, because it has a fixed in-neighbor.
    Hence, $Y$ covers $P$.
    Otherwise, consider the non-fixed vertex $v$ closest to $s$ on $P$.
    If $v = s$, then $v \in S$, and thus $v \in X$.
    If $v \neq s$, then the in-neighbor of $v$ is fixed, so $v \in X$ as well.
    Therefore, $X$ covers $P$, which completes the proof of the claim.
    \end{subproof}

  If Algorithm~\ref{alg:kDisjointPaths} terminates at Line~\ref{alg:line:cutSet}, then by Menger's theorem and Claim~\ref{cl:cutSet}, $D$ does not contain $k$ disjoint $S$\nobreakdash-$T$ dipaths, and no solution exists.
  Now we prove that there exists no solution if the algorithm terminates at Line~\ref{alg:line:certificate}.
  We need the following claim.
  \begin{claim}\label{cl:certificate}
    If there exists a subset $V' \subseteq V$ such that
    \begin{enumerate}[label={(\arabic*)}]
    \item each vertex $s \in V' \cap S$ has $\delta(s, V' \setminus \{s\}) \geq 2$,\label{enu:1}
    \item each vertex $t \in V' \cap T$ with $\varrho(t, V \setminus V') \geq 1$ has $\delta(t, V' \setminus \{t\}) \geq 1$,\label{enu:2}
    \item each vertex $v \in V' \setminus (S \cup T)$ has $\delta(v, V' \setminus \{v\}) \geq 1$, and\label{enu:3}
    \item each vertex $v \in V' \setminus (S \cup T)$ with $\varrho(v, V \setminus V') \geq 1$ has $\delta(v, V' \setminus \{v\}) \geq 2$,\label{enu:4}
    \end{enumerate}
    then no order exists in which the left-going arcs form $k$ disjoint $S$\nobreakdash-$T$ dipaths.
  \end{claim}
  \begin{subproof}
    Consider an order $\sigma$, and let $v'$ be the last vertex of $V'$ in $\sigma$.
    We argue that $v'$ violates its role in the subgraph of the left-going arcs:
    If $v' \in S$, then $\delta^\ell(v') \geq \delta(v', V' \setminus \{v'\}) \geq 2$ by~\ref{enu:1}, so $v'$ cannot be the first vertex of a dipath in the subgraph of left-going arcs.
    If $v' \in T$, then either $\varrho(v', V \setminus V') \geq 1$ and $\delta^\ell(v') \geq \delta(v', V' \setminus \{v'\}) \geq 1$ by~\ref{enu:2}, or $\varrho^r(v')\leq\varrho(v', V \setminus V') = 0$ and $\delta^\ell(v') = 0$, so $v'$ cannot be the last vertex of a dipath.
    If $v' \in V' \setminus (S \cup T)$, then either $\varrho(v', V \setminus V') \geq 1$ and $\delta^\ell(v') \geq \delta(v', V' \setminus \{v'\}) \geq 2$ by~\ref{enu:4}, or $\varrho^r(v')\leq\varrho(v', V \setminus V') = 0$ and $\delta^\ell(v') \geq \delta(v', V' \setminus \{v'\}) \geq 1$ by~\ref{enu:3}, so $v'$ cannot be a middle vertex or an isolated vertex in the dipath.
    Thus, no feasible order $\sigma$ exists, which completes the proof.
  \end{subproof}

  If the algorithm terminates at Line~\ref{alg:line:certificate}, meaning no vertex could be fixed in the final iteration, then the set of non-fixed vertices $V'$ satisfies the conditions of Claim~\ref{cl:certificate}.
  Indeed, each $v\in V'\setminus (S\cup T)$ with $\varrho(v, V\setminus V') \geq 1$ and each $v\in V'\cap S$ is in $X$, therefore $\delta(v,V') \geq 2$ (otherwise $v$ could be fixed in Line~\ref{alg:line:fixX}), so these vertices satisfy conditions~\ref{enu:4} and~\ref{enu:1} in Claim~\ref{cl:certificate}, respectively.
  Moreover, each vertex $t\in V'\cap T$ with $\varrho(t,V\setminus V') \geq 1$ is in $Y$, therefore, it has $\delta(t,V') \geq 1$ (otherwise it could be fixed in Line~\ref{alg:line:fix0}); and each $v\in V\setminus(S\cup T)$ has $\delta^{\ell}(v) \geq 1$ (otherwise it could be fixed either in Line~\ref{alg:line:fixX} or in Line~\ref{alg:line:fix0}), therefore, these vertices satisfy conditions~\ref{enu:2} and~\ref{enu:3} in Claim~\ref{cl:certificate}, respectively.
  This guarantees that no feasible order exists.
\end{proof}

Note that Claim~\ref{cl:certificate} and Menger's theorem together provide a necessary and sufficient condition for the existence of a vertex order in which the left-going arcs form $k$ disjoint $S$\nobreakdash-$T$ dipaths, where $k = |S| = |T|$.

Theorem~\ref{thm:kPathOrder} implies that we can decide whether there exists an ordering of the vertices such that the left-going arcs form a constant number of disjoint dipaths, because one can enumerate every possible $S$ and $T$ in polynomial time.
In contrast, the problem becomes NP-complete when the number of disjoint dipaths is not fixed, as proven in Corollary~\ref{cor:disjPathOrder}.
Moreover, using Theorem~\ref{thm:kPathOrder}, we can decide in polynomial time whether an ordering exists in which the left-going arcs form a dipath.
\begin{corollary}\label{cor:pathOrder}
  It can be decided in polynomial time whether the vertices of a digraph can be ordered such that the left-going arcs form a dipath.
  \FBOX
\end{corollary}
In contrast, partitioning into a dipath and an acyclic subgraph is NP-complete by Theorem~\ref{thm:pathAcyclic}.

\section{Open questions}
In Section~\ref{sec:Complexity}, we explored the complexity of the $(f,g)$\nobreakdash-bounded ordering problem under certain specific constraints.
We showed that the problem is NP-complete for the case where $f(v) = 1$ and $g(v) = \delta(v) - 2$ for each vertex $v \notin \{s,t\}$, where $s$ and $t$ are fixed as the first and last vertices.
This formulation is equivalent to finding an order in which each vertex, except for $s$ and $t$, has exactly one arc going out to the left and two arcs going out to the right.
However, the complexity of this problem remains open for undirected graphs, a question first posed in~\cite{kiraly2018acyclic}.

In Section~\ref{sec:dDist}, we examined the $d$\nobreakdash-distance $(-\infty, g)$\nobreakdash-bounded problem, showing that it is solvable in polynomial time when $d = |V| - k$ for some constant $k$.
However, it remains an open question whether there exists an FPT algorithm with parameter $(|V| - d)$.
Furthermore, we saw that for small values of $d$, for instance, when $d$ is constant, the problem becomes NP-complete.
The complexity is still unresolved for intermediate cases, such as when $d = \frac{|V|}{2}$.

One of the most intriguing open questions is the complexity of partitioning a digraph into an in-arborescence and an acyclic subgraph, or more generally, maximizing the size of an in-branching that covers all directed cycles.
In Section~\ref{sec:ArcPartition}, we showed that a similar problem, partitioning into an in-branching and an acyclic subgraph, is solvable in polynomial time.

It was also established that finding an order in which the left-going arcs form disjoint dipaths is NP-complete.
However, the problem becomes polynomial-time solvable when the number of disjoint dipaths is a fixed constant.
It remains an open question whether an FPT algorithm exists for this problem with the number of dipaths as the parameter.

In~\cite{bang2022complexity}, it was proven that decomposing a digraph into a directed 2-factor and an acyclic subgraph is NP-complete with respect to Turing reduction.
However, the complexity remains open for the case where the directed cycles are required to be disjoint but do not necessarily cover all vertices.
Other partitioning problems were posed for further research in~\cite{bang2022complexity}, which include covering all odd directed cycles with a perfect matching, or partitioning into a perfect matching and a subgraph containing an in-arborescence.

\section*{Acknowledgment}
This research has been implemented with the support provided by the Ministry of Innovation and Technology of Hungary from the National Research, Development and Innovation Fund, financed under the ELTE TKP 2021-NKTA-62 funding scheme, and by the Ministry of Innovation and Technology NRDI Office within the framework of the Artificial Intelligence National Laboratory Program, by the Lend\"ulet Programme of the Hungarian Academy of Sciences --- grant number LP2021-1/2021.
The first author was supported by the Ministry of Innovation and Technology of Hungary from the National Research, Development and Innovation Fund --- grant number ADVANCED 150556.
The second author was supported by the EK\" OP-24 University Excellence Scholarship Program of the Ministry for Culture and Innovation from the source of the National Research, Development and Innovation Fund.

\bibliographystyle{plain}
\bibliography{bibliography}

\begin{thebibliography}{10}

\bibitem{bang1991edge}
J.~Bang-Jensen.
\newblock Edge-disjoint in- and out-branchings in tournaments and related path
  problems.
\newblock {\em Journal of Combinatorial Theory, Series B}, 51(1):1--23, 1991.

\bibitem{bang2022complexity}
J.~Bang-Jensen, S.~Bessy, D.~Gon{\c{c}}alves, and L.~Picasarri-Arrieta.
\newblock Complexity of some arc-partition problems for digraphs.
\newblock {\em Theoretical Computer Science}, 928:167--182, 2022.

\bibitem{bang2015restricted}
J.~Bang-Jensen and C.~J. Casselgren.
\newblock Restricted cycle factors and arc-decompositions of digraphs.
\newblock {\em Discrete Applied Mathematics}, 193:80--93, 2015.

\bibitem{bang2020arc}
J.~Bang-Jensen, G.~Gutin, and A.~Yeo.
\newblock Arc-disjoint strong spanning subdigraphs of semicomplete
  compositions.
\newblock {\em Journal of Graph Theory}, 95(2):267--289, 2020.

\bibitem{bernath2015tractability}
A.~Bern{\'a}th and Z.~Kir{\'a}ly.
\newblock On the tractability of some natural packing, covering and
  partitioning problems.
\newblock {\em Discrete Applied Mathematics}, 180:25--35, 2015.

\bibitem{biedl2009complexity}
T.~Biedl, F.~J. Brandenburg, and X.~Deng.
\newblock On the complexity of crossings in permutations.
\newblock {\em Discrete Mathematics}, 309(7):1813--1823, 2009.

\bibitem{charikar2016approximating}
M.~Charikar, Y.~Naamad, and A.~Wirth.
\newblock On approximating target set selection.
\newblock In {\em Approximation, Randomization, and Combinatorial Optimization.
  Algorithms and Techniques (APPROX/RANDOM 2016)}. Schloss-Dagstuhl-Leibniz
  Zentrum f{\"u}r Informatik, 2016.

\bibitem{chen2009approximability}
N.~Chen.
\newblock On the approximability of influence in social networks.
\newblock {\em SIAM Journal on Discrete Mathematics}, 23(3):1400--1415, 2009.

\bibitem{cheriyan1994directeds}
J.~Cheriyan and J.~H. Reif.
\newblock Directed $s$-$t$ numberings, rubber bands, and testing digraph
  $k$-vertex connectivity.
\newblock {\em Combinatorica}, 14(4):435--451, 1994.

\bibitem{dinur2005hardness}
I.~Dinur and S.~Safra.
\newblock On the hardness of approximating minimum vertex cover.
\newblock {\em Annals of mathematics}, pages 439--485, 2005.

\bibitem{edmonds1973edge}
J.~Edmonds.
\newblock Edge-disjoint branchings.
\newblock {\em Combinatorial algorithms}, pages 91--96, 1973.

\bibitem{frank2011connections}
A.~Frank.
\newblock {\em Connections in Combinatorial Optimization}, volume~38.
\newblock Oxford University Press Oxford, 2011.

\bibitem{guruswami2008beating}
V.~Guruswami, R.~Manokaran, and P.~Raghavendra.
\newblock Beating the random ordering is hard: Inapproximability of maximum
  acyclic subgraph.
\newblock In {\em 2008 49th Annual IEEE Symposium on Foundations of Computer
  Science}, pages 573--582. IEEE, 2008.

\bibitem{kameda1973}
T.~Kameda and S.~Toida.
\newblock Efficient algorithms for determining an extremal tree of a graph.
\newblock In {\em 14th Annual Symposium on Switching and Automata Theory},
  pages 12--15, 1973.

\bibitem{kann1992approximability}
V.~Kann.
\newblock {\em On the approximability of NP-complete optimization problems}.
\newblock PhD thesis, Royal Institute of Technology Stockholm, 1992.

\bibitem{karp1972reducibility}
R.~M. Karp.
\newblock Reducibility among combinatorial problems.
\newblock In {\em Complexity of computer computations}, pages 85--103.
  Springer, 1972.

\bibitem{kemeny1959mathematics}
J.~G. Kemeny.
\newblock Mathematics without numbers.
\newblock {\em Daedalus}, 88(4):577--591, 1959.

\bibitem{kiraly2018acyclic}
Z.~Kir\'aly and D.~P\'alv\"olgyi.
\newblock Acyclic orientations with degree constraints.
\newblock arXiv:1806.03426, 2018.

\bibitem{kishi1969maximally}
G.~Kishi and Y.~Kajitani.
\newblock Maximally distant trees and principal partition of a linear graph.
\newblock {\em IEEE Transactions on Circuit Theory}, 16(3):323--330, 1969.

\bibitem{lempel1967algorithm}
A.~Lempel.
\newblock An algorithm for planarity testing of graphs.
\newblock In {\em Theory of Graphs: International Symposium.}, pages 215--232.
  Gorden and Breach, 1967.

\bibitem{lick1970k}
D.~R. Lick and A.~T. White.
\newblock $k$-degenerate graphs.
\newblock {\em Canadian Journal of Mathematics}, 22(5):1082--1096, 1970.

\bibitem{madarasi2020distance}
P.~Madarasi.
\newblock The distance matching problem.
\newblock In {\em International Symposium on Combinatorial Optimization}, pages
  202--213. Springer, 2020.

\bibitem{madarasi2021matchings}
P.~Madarasi.
\newblock Matchings under distance constraints {I.}
\newblock {\em Annals of Operations Research}, 305(1):137--161, 2021.

\bibitem{madarasi2024matchings}
P.~Madarasi.
\newblock Matchings under distance constraints {II.}
\newblock {\em Annals of Operations Research}, 332(1):303--327, 2024.

\bibitem{matula1983smallest}
D.~W. Matula and L.~L. Beck.
\newblock Smallest-last ordering and clustering and graph coloring algorithms.
\newblock {\em Journal of the ACM (JACM)}, 30(3):417--427, 1983.

\bibitem{montassier2012decomposing}
M.~Montassier, P.~Ossona~de Mendez, A.~Raspaud, and X.~Zhu.
\newblock Decomposing a graph into forests.
\newblock {\em Journal of Combinatorial Theory, Series B}, 102(1):38--52, 2012.

\bibitem{opatrny1979total}
J.~Opatrny.
\newblock Total ordering problem.
\newblock {\em SIAM Journal on Computing}, 8(1):111--114, 1979.

\bibitem{porschen2014xsat}
S.~Porschen, T.~Schmidt, E.~Speckenmeyer, and A.~Wotzlaw.
\newblock {XSAT and NAE-SAT of linear CNF classes}.
\newblock {\em Discrete Applied Mathematics}, 167:1--14, 2014.

\bibitem{wood2004bounded}
D.~R. Wood.
\newblock Bounded degree acyclic decompositions of digraphs.
\newblock {\em Journal of Combinatorial Theory, Series B}, 90(2):309--313,
  2004.

\bibitem{yang2018decomposing}
D.~Yang.
\newblock Decomposing a graph into forests and a matching.
\newblock {\em Journal of Combinatorial Theory, Series B}, 131:40--54, 2018.

\end{thebibliography}
\end{document}